\documentclass[12pt,reqno]{amsart}
\usepackage{amsmath,amssymb,extarrows}
\usepackage{url}
\usepackage{tikz,enumerate}
\usepackage{diagbox}
\usepackage{appendix}
\usepackage{epic}
\usepackage{comment}

\usepackage{longtable} 

\usepackage{float} 

\usepackage{cite}

\usepackage{hyperref}
\usepackage{array}

\usepackage{booktabs}

\allowdisplaybreaks[4]

\usepackage{makecell}
\usepackage{float} 
\usepackage{siunitx} 
\usepackage{multirow}
\usepackage[figuresright]{rotating}

\newtheorem{theorem}{Theorem}

\newtheorem{conj}[theorem]{Conjecture}
\newtheorem{lemma}[theorem]{Lemma}

\theoremstyle{definition}

\theoremstyle{remark}

\numberwithin{equation}{section}
\numberwithin{theorem}{section}
\numberwithin{defn}{section}

\DeclareMathOperator{\CT}{CT}

\newcommand{\padedvphantom}[3]{%
	\vtop{%
		\vbox{%
			\vspace*{#2}%
			\hbox{\vphantom{#1}}%
		}%
		\vspace*{#3}%
	}%
}

\begin{document}
\title[Some New Modular Rank Four Nahm Sums]
 {Some New Modular Rank Four Nahm Sums as Lift-dual of Rank Three Examples}

\author{Zhineng Cao and Liuquan Wang}

\address[Z.\ Cao]{School of Mathematics and Statistics, Wuhan University, Wuhan 430072, Hubei, People's Republic of China}
\email{zhncao@whu.edu.cn}

\address[L.\ Wang]{School of Mathematics and Statistics, Wuhan University, Wuhan 430072, Hubei, People's Republic of China}

\email{wanglq@whu.edu.cn;mathlqwang@163.com}



\begin{abstract}
We find nine new sets of rank four Nahm sums associated with nine different numeric matrices which are likely to be modular. They are discovered by applying the lift-dual operation to some modular rank three Nahm sums in the works of Zagier and the authors. We prove the modularity of four sets of these Nahm sums by establishing Rogers--Ramanujan type identities which express them as modular infinite products. We use various $q$-series techniques including the constant term method and Bailey pairs to prove these identities. Meanwhile, we present some conjectural identities expressing several Nahm sums as modular infinite products.
\end{abstract}

\maketitle

\section{Introduction}\label{sec-intro}
Throughout this paper we use standard $q$-series notation:
\begin{align}
(a;q)_0:=1, \quad (a;q)_n:=\prod\limits_{k=0}^{n-1}(1-aq^k), \quad (a;q)_\infty :=\prod\limits_{k=0}^\infty (1-aq^k),  \\
(a_1,\dots,a_m;q)_n:=(a_1;q)_n\cdots (a_m;q)_n, \quad n\in \mathbb{N}\cup \{\infty\}.
\end{align}
For brevity we also adopt the notation:
\begin{align}\label{J-defn}
J_m:=(q^{m};q^{m})_{\infty}, \quad J_{a,m}:=(q^{a},q^{m-a},q^{m};q^{m})_{\infty}.
\end{align}

Rogers--Ramanujan type identities are certain sum-to-product $q$-series identities which express $q$-hypergeometric series as infinite products. The study of them dates back to the famous Rogers--Ramanujan identities \cite{Rogers}:
\begin{align}\label{RR}
\sum_{n=0}^\infty \frac{q^{n^2}}{(q;q)_n}=\frac{1}{(q,q^4;q^5)_\infty},
\quad
\sum_{n=0}^\infty \frac{q^{n^2+n}}{(q;q)_n}=\frac{1}{(q^2,q^3;q^5)_\infty}.
\end{align}
Such kind of identities have significant implications in various areas such as combinatorics, number theory and mathematical physics. In particular, the product side of such identities usually indicates the modularity of the original series, which is not easy to be observed from the sum-side. One of the key problems linking the theory of $q$-series and modular forms is to determine the modularity of certain $q$-hypergeometric series.

With motivation from rational conformal field theories, Nahm \cite{Nahm1994,Nahmconf,Nahm07} considered a particular class of $q$-hypergeometric series called \emph{Nahm sum} or \emph{Nahm series}:
\begin{align}\label{eq-Nahm}
    f_{A,B,C}(q):=\sum_{n=(n_1,\dots,n_r)^\mathrm{T} \in \mathbb{N}^r}\frac{q^{\frac{1}{2}n^\mathrm{T}An+n^\mathrm{T}B+C}}{(q;q)_{n_1} \cdots (q;q)_{n_r}}.
\end{align}
Here $r$ is a positive integer, $A$ is a $r\times r$ positive definite matrix, $B$ is a $r$-dimensional column vector and $C$ is a rational scalar. Nahm proposed the problem to find all $A,B,C$ with rational entries so that $f_{A,B,C}(q)$ is modular, and such $(A,B,C)$ is usually referred to as a \emph{modular triple}.

Around 2007, Zagier \cite{Zagier} explicitly stated a conjecture attributed to Nahm, which provides some necessary and sufficient conditions on $A$ so that it is the matrix part of some modular triple $(A,B,C)$. This conjecture was proved in the rank one case by Zagier \cite{Zagier}. In fact, Zagier \cite{Zagier} showed that there are exactly seven rank one modular triples with $(A,B,C)$ being
\begin{equation}\label{rank1-exam}
\begin{split}
&(2,0,-1/60), ~~ (2,1,11/60), ~~ (1/2,0,-1/40), ~~ (1/2,1/2,1/40), \\
&\left(1,0,-1/48\right), \quad \left(1,1/2,1/24\right), \quad \left(1,-1/2,1/24\right).
\end{split}
\end{equation}
Vlasenko and Zwegers \cite{VZ} discovered some counterexamples to Nahm's conjecture in the rank two case. Recently, Calegari, Garoufalidis and Zagier \cite{CGZ} proved one direction of Nahm's conjecture.

As a full solution to Nahm's problem seems far from reach,  it deserves to find more and more explicit modular triples. In this direction, Zagier \cite{Zagier} provided 11 and 12 possible modular triples in the rank two and rank three cases, respectively. Their modularity have now all been confirmed by  Zagier \cite{Zagier}, Vlasenko--Zwegers \cite{VZ}, Cherednik--Feigin \cite{Feigin}, Wang \cite{Wang-rank2,Wang-rank3} and Cao--Rosengren--Wang \cite{CRW}. In their study of twisted modules of principal subspace vertex algebras, Calinescu--Milas--Penn \cite{CMP} conjectured that the rank $r$ tadpole Nahm sum associated with the $r\times r$ tadpole Cartan matrix is modular, and they proved the rank two case which coincides with one of Zagier's rank two examples \cite[Table 2]{Zagier}. Milas and Wang \cite{MW24} proved the rank three case and thereby proving the modularity of the dual example of Zagier's sixth rank three example. Recently, Shi and Wang \cite{Shi-Wang} proved the Calinescu--Milas--Penn conjecture for the cases $r=4,5$ and provided some new rank four and rank five modular triples.

Besides extensive search through computer programs, it is natural to ask whether we can generate new modular triples from known ones. Zagier \cite{Zagier} observed that there might exist some dual structure among modular triples. Let $(A,B,C)$ be a rank $r$ modular triple. We define its dual as the image of the operator:
\begin{align}
    \mathcal{D}:(A,B,C)\longmapsto (A^\star, B^\star, C^\star)=(A^{-1},A^{-1}B,\frac{1}{2}B^\mathrm{T} A^{-1}B-\frac{r}{24}-C).
\end{align}
As observed and suggested by Zagier \cite[p.\ 50, (f)]{Zagier}, it is very likely that $\mathcal{D}(A,B,C)$ is still a modular triple. Recently, Wang \cite{Wang-counterexample}  provided a counterexample for this expectation. For any positive integer $m$, Vlasenko and Zwegers \cite[Theorems 4.1 and 4.2]{VZ} presented some formulas revealing relations between certain rank $mr$ Nahm sums and  rank $r$ Nahm sums.

Based on Zagier's duality expectation, the authors \cite{Cao-Wang2024} presented an approach called lift-dual operation to find new rank $r+1$ modular triples from rank $r$ modular triples and present explicit examples in the case $r=2$. Namely, for any triple $(A,B,C)$ with
\begin{align*}
&A=\begin{pmatrix} a_1 & a_2 \\ a_2 & a_3\end{pmatrix}, \quad B=\begin{pmatrix} b_1 \\ b_2 \end{pmatrix} \quad \text{and} \quad C\in \mathbb{Q},
\end{align*}
the \emph{lifting operator} lifts $A$ to a $3\times 3$ matrix and $B$ to a three dimensional vector and keep the value of $C$:
\begin{align}
    \mathcal{L}: (A,B,C)\longmapsto (\widetilde{A},\widetilde{B},C)
\end{align}
where
\begin{align*}
&\widetilde{A}=\begin{pmatrix} a_1 & a_2+1 & a_1+a_2 \\ a_2+1 & a_3 & a_2+a_3 \\ a_1+a_2 & a_2+a_3 & a_1+2a_2+a_3 \end{pmatrix}, \quad \widetilde{B}=\begin{pmatrix} b_1 \\ b_2 \\ b_1+b_2\end{pmatrix}.
\end{align*}
Because of the identity \cite[p.\ 20]{Andrews1984}:
\begin{align}\label{Andrews-id}
\frac{1}{(q;q)_i(q;q)_j}=\sum_{k\geq 0} \frac{q^{(i-k)(j-k)}}{(q;q)_k(q;q)_{i-k}(q;q)_{j-k}},
\end{align}
we know that
\begin{align}\label{eq-lift-id}
    f_{A,B,C}(q)=f_{\widetilde{A},\widetilde{B},C}(q).
\end{align}
Hence if $(A,B,C)$ is a rank two modular triple, then $\mathcal{L}(A,B,C)$ is a rank three modular triple if $\widetilde{A}$ is positive definite. Zagier's observation then predicates that the dual triple $\mathcal{D}\mathcal{L}(A,B,C)$ is likely to be modular. The authors \cite{Cao-Wang2024} applied this lift-dual operator to all of Zagier's rank two examples \cite[Table 2]{Zagier}. As a result, they found three sets of new nontrivial rank three modular triples with matrix part being
\begin{align}
    A_1=\begin{pmatrix}
        \frac{1}{2}+m & \frac{1}{2}-m & -\frac{1}{2} \\
        \frac{1}{2}-m & \frac{1}{2}+m & -\frac{1}{2} \\
        -\frac{1}{2} & -\frac{1}{2} & 1
    \end{pmatrix}, \quad A_2=\begin{pmatrix}
        2 & -1 & -1 \\ -1 & 2 & 0 \\ -1 & 0 & 2
    \end{pmatrix}
\end{align}
and
\begin{align}
A_3=\begin{pmatrix}
        3/2 & -1/2 & -1/2 \\
        -1/2 & 3/2 & -1/2 \\
        -1/2 & -1/2 & 3/2
    \end{pmatrix},
\end{align}
respectively. The authors \cite[Eq.\ (6.1) and Table 2]{Cao-Wang2024} found more modular triples associated with $A_1$ when $m=\frac{1}{2}$. For convenience, we shall call the modular triples in \cite[Table 2, (4.3) and (1.15)]{Cao-Wang2024} associated with $A_1$ ($m=\frac{1}{2}$), $A_2$ and $A_3$ as the CW Examples 1, 2 and 3, respectively.

As a sequel to \cite{Cao-Wang2024}, we continue our investigation on finding new modular triples. This time we apply the lift-dual operation to some rank three modular triples to generate new rank four modular triples.

For any triple $(A,B,C)$ with
\begin{align*}
&A=\begin{pmatrix} a_1 & a_2 & a_3 \\ a_2 & a_4 & a_5 \\ a_3 & a_5 & a_6 \end{pmatrix}, \quad B=\begin{pmatrix} b_1 \\ b_2 \\ b_3 \end{pmatrix} \quad \text{and} \quad C\in \mathbb{Q},
\end{align*}
we define three lifting operators as follows:
\begin{align}\label{eq-lift}
    \mathcal{L}_{i}: (A,B,C) \longmapsto (\widetilde{A}_i,\widetilde{B}_i,C), \quad i=1,2,3
\end{align}
where
\begin{align}
&\widetilde{A}_1=\begin{pmatrix} a_1 & a_2 & a_3 & a_2+a_3 \\ a_2 & a_4 & a_5+1 & a_4+a_5 \\ a_3 & a_5+1 & a_6 & a_5+a_6 \\ a_2+a_3 & a_4+a_5 & a_5+a_6 & a_4+2a_5+a_6 \end{pmatrix}, \quad \widetilde{B}_1=\begin{pmatrix} b_1 \\ b_2 \\ b_3 \\ b_2+b_3\end{pmatrix}, \label{3-jk} \\
&\widetilde{A}_2=\begin{pmatrix} a_1 & a_2 & a_3+1 & a_1+a_3 \\ a_2 & a_4 & a_5 & a_2+a_5 \\ a_3+1 & a_5 & a_6 & a_3+a_6 \\ a_1+a_3 & a_2+a_5 & a_3+a_6 & a_1+2a_3+a_6 \end{pmatrix}, \quad \widetilde{B}_2=\begin{pmatrix} b_1 \\ b_2 \\ b_3 \\ b_1+b_3\end{pmatrix}, \label{3-ik} \\
&\widetilde{A}_3=\begin{pmatrix} a_1 & a_2+1 & a_3 & a_1+a_2 \\ a_2+1 & a_4 & a_5 & a_2+a_4 \\ a_3 & a_5 & a_6 & a_3+a_5 \\ a_1+a_2 & a_2+a_4 & a_3+a_5 & a_1+2a_2+a_4 \end{pmatrix}, \quad \widetilde{B}_3=\begin{pmatrix} b_1 \\ b_2 \\ b_3 \\ b_1+b_2\end{pmatrix}. \label{3-ij}
\end{align}
For convenience, we shall use $\mathcal{L}_i(A)$ and $\mathcal{L}_i(B)$ to denote $\widetilde{A}_i$ and $\widetilde{B}_i$ in \eqref{eq-lift} and agree that $\mathcal{L}_i(C)=C$.
For the same reason as above (see \eqref{Andrews-id}), we have
\begin{align}
f_{A,B,C}(q)=f_{\widetilde{A}_i,\widetilde{B}_i,C}(q).
\end{align}
Hence if $(A,B,C)$ is a modular triple, then so is any of its lift $(\widetilde{A}_i,\widetilde{B}_i,C)$ ($i=1,2,3$) whenever $\widetilde{A}_i$ is positive definite. Zagier's duality observation then inspires us to consider the modularity of the dual triples $\mathcal{D}\mathcal{L}_i(A,B,C)$ ($i=1,2,3$) which we shall call as the \emph{lift-dual} of the original triple $(A,B,C)$.

In this paper, we mainly apply the lift-dual operation to the rank three modular triples in the works of Zagier \cite{Zagier} and the authors \cite{Cao-Wang2024} to see whether we can get rank four modular triples. Note that some matrices in \cite{Zagier,Cao-Wang2024} contain free parameters and thus generate infinite families of modular Nahm sums. To save space, here we only consider the triples $(A,B,C)$ with $A$ being a purely numeric matrix. The cases where $A$ contains free parameters are more complicated and left to a future work. As in \cite{Wang-rank3}, we label Zagier's rank three examples as Examples 1--12 according to the order of their appearance in \cite[Table 3]{Zagier}. We shall consider Examples 4--12 among Zagier's twelve rank three examples \cite[Table 3]{Zagier} and the aforementioned three sets of modular triples in the work of Cao--Wang \cite[Table 2, (4.3) and (1.15)]{Cao-Wang2024}.  The matrix part and the determinants of their lift are listed in Table \ref{tab-lift}. Note that the lifts of other rank three numeric matrices in \cite[(1.14), (4.2) and (6.1)]{Cao-Wang2024} are not positive definite and hence neglected.
{\small
\begin{table}[htbp]
\centering
 \caption{Matrix part of Zagier's rank three Examples 4--12 and the CW Examples 1--3}\label{tab-lift}
 \begin{tabular}{|c|c|ccc|}
\hline
      Reference & $A$ & $\det \mathcal{L}_1(A)$ & $\det \mathcal{L}_2(A)$ & $\det \mathcal{L}_3(A)$ \\
\hline 
\padedvphantom{I}{2ex}{2ex}
Zagier's Exam.\ 4 & $\left(\begin{smallmatrix}
    2 & 1 & -1 \\ 1 & 1 & 0 \\ -1 & 0 &1
\end{smallmatrix}\right)$
    & $-4$
    & $0$ & $-4$ \\
    \hline  
    \padedvphantom{I}{2ex}{2ex}
Zagier's Exam.\ 5 & $\left(\begin{smallmatrix}
    2 & 1 & 1 \\ 1 & 1 & 0 \\ 1 & 0 & 1
\end{smallmatrix}\right)$ & $0$  & $-4$  & $-4$  \\
\hline   
\padedvphantom{I}{2ex}{2ex}
Zagier's Exam.\ 6 & $\left(\begin{smallmatrix}
    3 & 2 & 1 \\ 2 & 2 & 1 \\ 1 & 1 & 1
\end{smallmatrix}\right)$ & $-4$ & $-1$ & $-3$ \\
\hline   
\padedvphantom{I}{2ex}{2ex}
Zagier's Exam.\ 7 & $\left(\begin{smallmatrix}  2 & 1 & 1 \\  1 & 2 & 0 \\  1 & 0 & 2  \end{smallmatrix}\right)$  & 4 & $-3$ & $-3$\\
\hline  
\padedvphantom{I}{2ex}{2ex}
Zagier's Exam.\ 8 & $\left(\begin{smallmatrix}
            4 & 2 & 1 \\ 2 & 2 & 0 \\ 1 & 0 & 1
        \end{smallmatrix}\right)$ & 1 & $-6$  & $-5$\\
\hline  
\padedvphantom{I}{2ex}{2ex}
Zagier's Exam.\ 9 & $\left(\begin{smallmatrix} 6 & 4 & 2 \\ 4 & 4 & 2 \\ 2 & 2 & 2 \end{smallmatrix}\right)$ & $-8$ & 4 & $-4$ \\
\hline  
\padedvphantom{I}{2ex}{2ex}
Zagier's Exam.\ 10 & $\left(\begin{smallmatrix}
    4 & 2 & 2 \\ 2 & 2 & 1 \\ 2 & 1 & 2
\end{smallmatrix}\right)$ & $0$ & $-3$ & $-3$ \\
\hline   
\padedvphantom{I}{2ex}{2ex}
Zagier's Exam.\ 11 & $\left(\begin{smallmatrix} 4 & 2 & -1 \\ 2 & 2 & -1 \\ -1 & -1 & 1 \end{smallmatrix}\right)$ & 1 & $-1$ & $-2$ \\
\hline  
\padedvphantom{I}{2ex}{2ex}
Zagier's Exam.\ 12 & $\left(\begin{smallmatrix} 8 & 4 & 1 \\ 4 & 3 & 0 \\ 1 & 0 & 1 \end{smallmatrix}\right)$ & 3 & $-7$ & $-8$ \\
\hline   
\padedvphantom{I}{2ex}{2ex}
\makecell{CW Exam.\ 1\\ \cite[(6.31), Table 2]{Cao-Wang2024}}& $\left(\begin{smallmatrix} 1 & 0 & -1/2 \\  0 & 1 & -1/2 \\ -1/2 & -1/2 & 1 \end{smallmatrix}\right)$ & $1/4$  & $1/4$ & 0  \\
\hline   
\padedvphantom{I}{2ex}{2ex}
\makecell{CW Exam.\ 2 \\ \cite[(4.3)]{Cao-Wang2024}} & $\left(\begin{smallmatrix} 2 & -1 & -1 \\ -1 & 2 & 0 \\ -1 & 0 & 2 \end{smallmatrix}\right)$ & 4 & 5   & 5 \\
\hline  
\padedvphantom{I}{2ex}{2ex}
\makecell{CW Exam.\ 3 \\ \cite[(1.15)]{Cao-Wang2024}}  &  $\left(\begin{smallmatrix} 3/2 & -1/2 & -1/2  \\  -1/2 & 3/2 & -1/2  \\ -1/2 & -1/2 & 3/2  \\  \end{smallmatrix}\right)$ & 2 & 2 & 2 \\
\hline
\end{tabular}
\end{table}
}

As can be seen from Table \ref{tab-lift}, the lift of Zagier's Examples 4, 5, 6, 10 are no longer positive definite matrices and considering their dual does not make sense. It remains to consider Zagier's Examples 7, 8, 9, 11 and 12 as well as the CW Examples 1--3. Altogether, they generate nine different sets of rank four Nahm sums corresponding to nine different rank four matrices. Here we regard two Nahm sums as the same if we can obtain one from the other by interchanging some of the summation indices.   Specifically speaking, the CW Example 2 generates two different sets of Nahm sums since the matrix lifted via $\mathcal{L}_2$ can be transformed to the matrix lifted via $\mathcal{L}_3$ by interchanging the second and third rows/columns. In contrast, the other seven examples all generate one set of rank four Nahm sums. We are able to prove the modularity of four of them. Namely, by establishing the corresponding Rogers--Ramanujan type identities, we confirm modularity of the lift-dual of Zagier's Example 7, the CW Examples 1 and 3 and the lift-dual triple generated using $\mathcal{L}_1$ from the CW Example 2. Among them the lift-dual of the CW Example 1 is particularly complicated, and we will prove some identities such as (see Theorem \ref{CW-thm6.2-ij3-in})
\begin{align}
& \sum_{i,j,k,l\geq 0} \frac{q^{2i^2+2j^2+2k^2+2l^2-2ij-2ik-2il+2jk+2jl-i+k-l}}{(q^2;q^2)_i(q^2;q^2)_j(q^2;q^2)_k(q^2;q^2)_l}=2\frac{J_2^2}{J_1^2}, \label{intro-thm3-3} \\
&    \sum_{i,j,k,l\geq 0} \frac{q^{i^2+j^2+k^2+l^2-ij-ik-il+jk+jl+j+l}}{(q;q)_i(q;q)_j(q;q)_k(q;q)_l}=\frac{1}{3}\frac{J_{6}^{2}}{J_{1}^{2}J_{12}}\Big(\frac{J_{2}^{5}J_{2,12}}{J_{1}^{2}J_{4}^{2}J_{1,6}}+\frac{2J_{4}^{3}}{J_{2}^{2}} \Big), \label{intro-thm3-10}\\
   &\sum_{i,j,k,l\geq 0} \frac{q^{i^2+j^2+k^2+l^2-ij-ik-il+jk+jl+i-k}}{(q;q)_i(q;q)_j(q;q)_k(q;q)_l}=
   \frac{2}{3}\frac{J_{6}^{2}}{J_{1}^{2}J_{12}}\Big(\frac{J_{2}^{5}J_{2,12}}{J_{1}^{2}J_{4}^{2}J_{1,6}}+\frac{2J_{4}^{3}}{J_{2}^{2}} \Big), \label{intro-thm3-11} \\
  & \sum_{i,j,k,l\geq 0} \frac{q^{i^2+j^2+k^2+l^2-ij-ik-il+jk+jl+i-j-k-l}}{(q;q)_i(q;q)_j(q;q)_k(q;q)_l}=
   6\frac{J_3^3}{J_1^3}. \label{intro-thm3-12}
\end{align}
The proof of \eqref{intro-thm3-3}--\eqref{intro-thm3-12} requires intricate use of the constant term method applied to two different variables. Moreover, the proofs of \eqref{intro-thm3-10}--\eqref{intro-thm3-12} also rely on the method of Bailey pairs. This already shows that proving the modularity of the Nahm sums generated by the lift-dual approach can sometimes be quite complicated and difficult. At this stage, we are not able to prove the modularity of the lift-dual of Zagier's Example 8, 9, 11 and 12 as well as the lift-dual triples generated using $\mathcal{L}_2$ (equivalently $\mathcal{L}_3$) from the CW Example 2. We present several conjectural formulas for some particular cases in these conjectural examples. In particular, we find three new (conjectural) modular cases for rank four Nahm sums associated with the tadpole Cartan matrix besides those discussed by Shi and Wang \cite{Shi-Wang}.

The rest of this paper is organized as follows. In Section \ref{sec-pre} we collect some auxiliary identities and review some basic knowledge of Bailey pairs. These will be employed in establishing the Rogers--Ramanujan type identities for our rank four Nahm sums. In Sections \ref{sec-exam} and \ref{sec-CW} we discuss modularity of the rank four Nahm sums generated from the rank three examples of Zagier \cite{Zagier} and the authors \cite{Cao-Wang2024}, respectively.

\section{Preliminaries}\label{sec-pre}

We first recall some basic identities from the theory of $q$-series. The $q$-binomial theorem \cite[Theorem 2.1]{Andrews1998} asserts that
\begin{align}\label{q-binomial}
\sum_{n=0}^\infty \frac{(a;q)_n}{(q;q)_n}z^n=\frac{(az;q)_\infty}{(z;q)_\infty}.
\end{align}
As its consequence, we have Euler's $q$-exponential identities \cite[Corollary 2.2]{Andrews1998}
\begin{align}
\sum_{n=0}^\infty \frac{z^n}{(q;q)_n}&=\frac{1}{(z;q)_\infty}, \quad |z|<1, \label{Euler1} \\
\sum_{n=0}^\infty \frac{z^nq^{\frac{n^2-n}{2}}}{(q;q)_n}&=(-z;q)_\infty.  \label{Euler2}
\end{align}

The Jacobi triple product identity \cite[Theorem 2.8]{Andrews1998} states that
\begin{align}\label{Jacobi}
(z,q/z,q;q)_\infty=
\sum_{n=-\infty}^{{\infty}} (-z)^n q^{\frac{n^2-n}{2}}.
\end{align}

The basic hypergeometric series ${}_r\phi_s$ is defined as:
$${}_r\phi_s\bigg(\genfrac{}{}{0pt}{} {a_1,\dots,a_r}{b_1,\dots,b_s};q,z  \bigg):=\sum_{n=0}^\infty \frac{(a_1,\dots,a_r;q)_n}{(q,b_1,\dots,b_s;q)_n}\Big((-1)^nq^{\binom{n}{2}} \Big)^{1+s-r}z^n.$$
Occasionally we use the following without mention to simplify some $q$-products:
\begin{align}
&(aq^{-n};q)_n=(q/a;q)_n(-a/q)^nq^{-\frac{n^2-n}{2}}.
\end{align}

We recall several ${}_r\phi_s$ formulas that will be used in our proofs. The $q$-Gauss summation formula \cite[(1.5.1)]{Gasper-Rahman} states that
\begin{align} \label{Gauss}
{}_2\phi_1\bigg(\genfrac{}{}{0pt}{} {a,b}{c};q,\frac{c}{ab}  \bigg)=\frac{(c/a,c/b;q)_\infty}{(c,c/ab;q)_\infty},  \quad \left| \frac{c}{ab} \right|<1.
\end{align}
As a $q$-analogue of Bailey's ${}_2F_{1}(-1)$ sum  we have \cite[\uppercase\expandafter{\romannumeral2}.10]{Gasper-Rahman}
\begin{align} \label{Bailey's}
{}_2\phi_2\bigg(\genfrac{}{}{0pt}{} {a,q/a}{-q,b};q,-b  \bigg)=\frac{(ab,bq/a;q^2)_\infty}{(b;q)_\infty}.
\end{align}
A $q$-analogue of Gauss' ${}_2F_{1}(-1)$ sum is given by  \cite[\uppercase\expandafter{\romannumeral2}.11]{Gasper-Rahman}
\begin{align} \label{Gauss'}
{}_2\phi_2\bigg(\genfrac{}{}{0pt}{} {a^2,b^2}{abq^{1/2},-abq^{1/2}};q,-q  \bigg)=\frac{(a^2q,b^2q;q^2)_\infty}{(q,a^2b^2q;q^2)_\infty}.
\end{align}

Slater \cite{Slater} provided a famous list containing 130 Rogers--Ramanujan type identities, and some of them will be employed in our proofs. For convenience, we use the label (S.~$n$) to denote the $n$-th identity in Slater's list \cite{Slater}. Here are the single sum Rogers--Ramanujan type identities we need:
\begin{align}
&\sum_{n\geq 0} \frac{q^{2n^2+n}}{(q;q)_{2n+1}}=\frac{J_{2}}{J_{1}}, \quad \text{(S.\ 9)} \label{S. 9}  \\
&\sum_{n\geq 0} \frac{q^{n^2+n}(-q^2;q^2)_n}{(q;q)_{2n+1}}=\frac{J_4J_{6}^{2}J_{2,12}}{J_{2}^{2}J_{1,6}J_{12}}, \quad \text{(S.\ 28)} \label{S. 28}\\
&\sum_{n\geq 0} \frac{q^{n^2}(-q;q^2)_n}{(q;q)_{2n}}=\frac{J_{6}^{2}}{J_{1}J_{12}}, \quad \text{(S. 29)} \label{S. 29}\\
&\sum_{n\geq 0} \frac{q^{2n(n+1)}}{(q^2;q^2)_n(-q;q)_{2n+1}}=\frac{J_{1,7}}{J_2}, \quad \text{(Rogers \cite[p.\ 331 (6)]{Rogers1937}; S.\ 31)} \label{S31} \\
&\sum_{n\geq 0} \frac{q^{2n(n+1)}}{(q^2;q^2)_n(-q;q)_{2n}}=\frac{J_{2,7}}{J_2}, \quad \text{(Rogers \cite[p.\ 342]{Rogers}; S.\ 32)}  \label{S32}\\
&\sum_{n\geq 0} \frac{q^{2n^2}}{(q^2;q^2)_n(-q;q)_{2n}}=\frac{J_{3,7}}{J_2}, \quad \text{(Rogers \cite[p.\ 339]{Rogers}; S.\ 33)}  \label{S33} \\
&\sum_{n\geq 0} \frac{q^{2n^2+2n}}{(q;q)_{2n+1}}=\frac{J_{3,8}J_{2,16}}{J_{1}J_{16}}, \quad \text{(S.\ 38)} \label{S. 38}\\
&\sum_{n\geq 0} \frac{q^{2n^2}}{(q;q)_{2n}}=\frac{J_{1,8}J_{6,16}}{J_{1}J_{16}}, \quad \text{(S.\ 39)} \label{S. 39}\\
&\sum_{n\geq 0} \frac{q^{(3n^2-n)/2}(-q;q)_n}{(q;q)_{2n}}=\frac{q^{(3n^2-n)/2}}{(q;q^2)_{n}(q;q)_n}=\frac{J_{4,10}}{J_1}, \quad \text{(S.\ 46)} \label{S. 46}\\
&\sum_{n\geq 0} \frac{q^{n^2+n}(-1;q^2)_n}{(q;q)_{2n}}=\frac{J_{5,12}-qJ_{1,12}}{J_1}, \quad \text{(S.\ 48)} \label{S. 48}\\
&\sum_{n\geq 0} \frac{q^{n^2+2n}(-q;q^2)_n}{(q;q)_{2n+1}}=\frac{J_{2,12}}{J_1}, \quad \text{(S.\ 50)} \label{S. 50}\\
&\sum_{n\geq 0} \frac{q^{n(n+2)}}{(q;q^2)_{n+1}(q;q)_n}=\frac{J_{2,14}}{J_1}, \quad \text{(Rogers \cite[p.\ 329 (1)]{Rogers1937}; S.\ 59)}  \label{S59} \\
&\sum_{n=0}^\infty \frac{q^{n(n+1)}}{(q;q^2)_{n+1}(q;q)_n}=\frac{J_{4,14}}{J_1}, \quad \text{(Rogers \cite[p.\ 329 (1)]{Rogers1937}; S.\ 60)}  \label{S60} \\
&\sum_{n=0}^\infty \frac{q^{n^2}}{(q;q^2)_n(q;q)_n}=\frac{J_{6,14}}{J_1}, \quad \text{(Rogers \cite[p.\ 341, Ex.\ 2]{Rogers}; S.\ 61)} \label{S61}\\
&\sum_{n\geq 0} \frac{q^{(3n^2+n)/2}(-q;q)_n}{(q;q)_{2n+1}}=\frac{J_{4,10}}{J_1}, \quad \text{(S.\ 62)} \label{S. 62}\\
&\sum_{n\geq 0} \frac{q^{3(n^2+n)/2}(-q;q)_n}{(q;q)_{2n+1}}=\frac{J_{2,10}}{J_1}, \quad \text{(S.\ 63)} \label{S. 63}\\
&\sum_{n\geq 0} \frac{q^{2n^2-n}}{(q;q)_{2n}}=\frac{J_{2}}{J_{1}}, \quad \text{(S.\ 85)} \label{S. 85}\\
&\sum_{n\geq 0} \frac{q^{3n^2-2n}(-q;q^2)_n}{(q^2;q^2)_{2n}}=\frac{J_{2}J_{3,10}J_{4,20}}{J_{1}J_{4}J_{20}}, \quad \text{(S.\ 95)} \label{S. 95}\\
&\sum_{n\geq 0} \frac{q^{3n^2+2n}(-q;q^2)_{n+1}}{(q^2;q^2)_{2n+1}}=\frac{J_{2}J_{3,10}J_{4,20}}{J_{1}J_{4}J_{20}}, \quad \text{(\cite[(A.\ 97)]{Sills})} \label{A. 97}\\
&\sum_{n\geq 0} \frac{q^{3n^2}(-q;q^2)_n}{(q^2;q^2)_{2n}}=\frac{J_{2}J_{1,10}J_{8,20}}{J_{1}J_{4}J_{20}}. \quad \text{(\cite[(A.\ 100)]{Sills})} \label{A. 100}
\end{align}

Similar to the works \cite{Wang-rank3,Cao-Wang2024}, we also need Bailey pairs to prove some identities we find. A pair of sequences $(\alpha_n(a;q),\beta_n(a;q))$ is called a Bailey pair relative to $a$ if for all $n\geq 0$,
 \begin{align}\label{defn-BP}
     \beta_n(a;q)=\sum_{k=0}^n\frac{\alpha_k(a;q)}{(q;q)_{n-k}(aq;q)_{n+k}}.
 \end{align}

\begin{lemma}[Bailey's lemma]
Suppose that $(\alpha_{n}(a;q),\beta_{n}(a;q))$ is a Bailey pair relative to $a$. Then $(\alpha_{n}'(a;q),\beta_{n}'(a;q))$ is another Bailey pair relative to $a$, where
\begin{align}\label{Bailey's lemma}
&\alpha_{n}'(a;q)=\frac{(\rho_{1},\rho_{2};q)_{n}}{(aq/\rho_{1},aq/\rho_{2};q)_{n}}\left( \frac{aq}{\rho_{1}\rho_{2}}\right)^{n}\alpha_{n}(a;q),  \\
&\beta_{n}'(a;q)=\sum_{k=0}^{n}\frac{(\rho_{1},\rho_{2};q)_{k}(aq/\rho_{1}\rho_{2};q)_{n-k}}{(aq/\rho_{1},aq/\rho_{2};q)_{n}(q;q)_{n-k}}\left( \frac{aq}{\rho_{1}\rho_{2}}\right)^{k}\beta_{k}(a;q).
\end{align}
\end{lemma}
Equivalently, if $(\alpha_{n}(a;q), \beta_{n}(a;q))$ is a Bailey pair, then
\begin{align}\label{Bailey's lemma-1}
&\frac{1}{(aq/\rho_{1},aq/\rho_{2};q)_{n}}\sum_{k=0}^{n}\frac{(\rho_{1},\rho_{2};q)_{k}(aq/\rho_{1}\rho_{2};q)_{n-k}}{(q;q)_{n-k}}\left( \frac{aq} {\rho_{1}\rho_{2}}\right)^{k}\beta_{k}(a;q) \nonumber\\
&=\sum_{r=0}^{n}\frac{(\rho_{1},\rho_{2};q)_{r}}{(q;q)_{n-r}(aq;q)_{n+r}(aq/\rho_{1},aq/\rho_{2};q)_{r}}\left( \frac{aq}{\rho_{1}\rho_{2}}\right)^{r}\alpha_{r}(a;q).
\end{align}

Letting $\rho_{1},\rho_{2}\rightarrow\infty$, we obtain the Bailey pair \cite[Eq. (S1)]{Bressoud2000}:
\begin{align}\label{Bailey's lemma-1.1}
\alpha_{n}'(a;q)=a^nq^{n^2}\alpha_{n}(a;q), \quad
\beta_{n}'(a;q)=\sum_{r=0}^{n}\frac{a^rq^{r^2}}{(q;q)_{n-r}}\beta_{r}(a;q).
\end{align}
If we further let $n\rightarrow\infty$, then we deduce from \eqref{Bailey's lemma-1} that
\begin{align}\label{Bailey's lemma-1.2}
\sum_{n=0}^{\infty}a^nq^{n^2}\beta_{n}(a;q)=\frac{1}{(aq;q)_{\infty}}\sum_{n=0}^{\infty}a^nq^{n^2}\alpha_{n}(a;q).
\end{align}

The following result of Lovejoy \cite[p.~1510]{Lovejoy2004} allows us to obtain a Bailey pair relative to $aq$ from a Bailey pair relative to $a$.
\begin{lemma}
If $(\alpha_{n}(a;q), \beta_{n}(a;q))$ is a Bailey pair relative to $a$, then  $(\alpha_{n}', \beta_{n}')$ is a Bailey pair relative to $aq$ where
\begin{align}\label{Lovejoy-a-aq}
\alpha_{n}'(aq;q)&=\frac{(1-aq^{2n+1})(aq/b)_{n}(-b)^{n}q^{n(n-1)/2}}{(1-aq)(bq;q)_{n}}\sum_{r=0}^{n}\frac{(b;q)_{r}}{(aq/b;q)_{r}}\nonumber\\
&\quad\times (-b)^{-r}q^{-r(r-1)/2}\alpha_{r}(a;q), \\
\beta_{n}'(aq;q)&=\frac{(b;q)_{n}}{(bq;q)_{n}}\beta_{n}(a;q). \nonumber
\end{align}
\end{lemma}

In particular, when $b\rightarrow 0$ we obtain the Bailey pair:
\begin{equation}\label{Lovejoy-a-aq-1}
\begin{split}
\alpha_{n}'(aq;q)&=\frac{(1-aq^{2n+1})a^{n}q^{n^2}}{1-aq}\sum_{r=0}^{n}a^{-r}q^{-r^2}\alpha_{r}(a;q),  \\
\beta_{n}'(aq;q)&=\beta_{n}(a;q).
\end{split}
\end{equation}

When $b\rightarrow \infty$ we obtain the Bailey pair:
\begin{equation}\label{W-BP-a-aq}
    \begin{split}
\alpha_n'(aq;q)=\frac{1-aq^{2n+1}}{1-aq}q^{-n}\sum_{r=0}^n \alpha_r(a;q), \quad \beta_n'(aq;q)=q^{-n}\beta_n(a;q).
    \end{split}
\end{equation}

The following result of Mc Laughlin \cite{McLaughlin2018} allows us to obtain a Bailey pair relative to $a/q$ from a Bailey pair relative to $a$.
\begin{lemma}
If $(\alpha_{n}(a;q), \beta_{n}(a;q))$ is a Bailey pair relative to $a$, then  $(\alpha_{n}', \beta_{n}')$ is a Bailey pair relative to $a/q$ where
\begin{align}\label{McLaughlin-a-a/q}
\alpha_{0}'(a/q;q)&=\alpha_{0}(a;q), \quad \alpha_{n}'(a/q;q)=(1-a)\left(\frac{\alpha_{n}(a;q)}{1-aq^{2n}}-\frac{aq^{2n-2}\alpha_{n-1}(a;q)}{1-aq^{2n-2}} \right), \nonumber\\
\beta_{n}'(a/q;q)&=\beta_{n}(a;q).
\end{align}
\end{lemma}

Another key approach in our proof of identities is the constant term method. For any series
$$f(z_1,z_2,\dots,z_k)=\sum_{n_1,n_2,\dots,n_k\in \mathbb{Z}}a_{n_1n_2\dots n_k}z_1^{n_1}z_2^{n_2}\cdots z_k^{n_k},$$
we define the constant term extractor with respect to the variables $z_1,z_2,\dots,z_k$ as
$$\CT_{z_1,z_2,\dots,z_k}f(z_1,z_2,\dots,z_k)=a_{00\dots0}.$$
It is easy to see that for any nonzero integers $e_1,e_2,\dots,e_k$ we have
\begin{align}
\CT_{z_1,z_2,\dots,z_k}f(z_1,z_2,\dots,z_k)=\CT_{z_1,z_2,\dots,z_k}f(z_1^{e_1},z_2^{e_2},\dots,z_k^{e_k}).
\end{align}

Finally, we recall the strategy explained in \cite{Wang-rank2,Wang-rank3} to justify modularity of Nahm sums. By establishing Rogers--Ramanujan type identities, we will express a Nahm sum $f_{A,B,C}(q)$ as
\begin{align}
    f_{A,B,C}(q)=q^Cf_{A,B,0}(q)=q^C(f_1(q)+f_2(q)+\cdots+f_n(q)).
\end{align}
Here $f_i(q)$ ($i=1,2,\dots,n$) are infinite products expressed by $J_m$ and $J_{a,m}$ defined in \eqref{J-defn}. It is known that both $q^{m/24}J_{m}$ and $q^{m/24+mP_{2}(a/m)/2}J_{a,m}$ are modular forms of weight 1/2, where $P_{2}(x)=x^2-x+1/6$. Hence it is easy to find the unique value $C_i$ so that $q^{C_i}f_{i}(q)$ is modular. When $C_1=C_2=\cdots=C_n=C$ and the functions $f_i(q)$ $(i=1,2,\dots,n$) all share the same weight $k$, then $f_{A,B,C}(q)$ is modular of weight $k$. Since this process is easy and routine, we will not mention the calculations of the values $C$.

\section{Lift-dual of Zagier's rank three examples}\label{sec-exam}

\subsection{Lift-dual of Zagier's Example 7}
 We list the modular triples $(A,B,C)$ in Zagier's Example 7 (see \cite[Table 3]{Zagier} or \cite[Example 7]{Wang-rank3}) and their lift and lift-dual in Table \ref{Wang-3-Exam7}. From Table \ref{tab-lift} we see that only the lift via $\mathcal{L}_1$ is meaningful, and we checked that $\mathcal{L}_1(A)$ is indeed positive definite.
{\small
\begin{table}[htbp]
\centering
\caption{Modular triples $(A,B,C)$ in Zagier's Example 7 and its lift and lift-dual via $\mathcal{L}_1$}
    \label{Wang-3-Exam7}
    \begin{tabular}{|c|ccccccc|}
  \hline 
  \padedvphantom{I}{2ex}{2ex}
   $A$ & \multicolumn{7}{c|}{$\left(\begin{smallmatrix}  2 & 1 & 1 \\  1 & 2 & 0 \\  1 & 0 & 2  \end{smallmatrix}\right)$} \\
   \hline 
   \padedvphantom{I}{2ex}{2ex}
        $B$ & $\left(\begin{smallmatrix} 0 \\ b \\ -b \end{smallmatrix}\right)$ & $\left(\begin{smallmatrix} 1/2 \\ 0 \\ 1 \end{smallmatrix}\right)$ & $\left(\begin{smallmatrix} 1/2 \\ 1 \\ 0 \end{smallmatrix}\right)$ & $\left(\begin{smallmatrix} 1 \\ 0 \\ 1 \end{smallmatrix}\right)$  &
        $\left(\begin{smallmatrix} 1 \\ 1 \\ 0 \end{smallmatrix}\right)$  &
        $\left(\begin{smallmatrix} -1/2 \\ 0 \\ 0\end{smallmatrix}\right)$ & $\left(\begin{smallmatrix} -1 \\ -1/2 \\ -1/2 \end{smallmatrix}\right)$  \\
         \padedvphantom{I}{1ex}{1ex}
        $C$ & {\tiny $(6b^2-1)/{24}$} & {\tiny $7/48$}   & {\tiny $7/48$} & {\tiny $5/24$} &  {\tiny $5/24$} & {\tiny $-1/48$}  &  {\tiny $1/48$}   \\
    \hline 
    \padedvphantom{I}{3ex}{3ex}
      $\mathcal{L}_1(A)$ & \multicolumn{7}{c|}{$\left(\begin{smallmatrix} 2 & 1 & 1 & 2 \\  1 & 2 & 1 & 2 \\ 1 & 1 & 2 & 2 \\ 2 & 2 & 2 & 4  \end{smallmatrix}\right)$} \\
   \hline 
   \padedvphantom{I}{3ex}{3ex}
        $\mathcal{L}_1(B)$ & $\left(\begin{smallmatrix}0 \\ b \\ -b \\ 0 \end{smallmatrix}\right)$ & $\left(\begin{smallmatrix}  1/2 \\ 0 \\ 1 \\ 1  \end{smallmatrix}\right)$ & $\left(\begin{smallmatrix} 1/2 \\ 1 \\ 0 \\ 1  \end{smallmatrix}\right)$ & $\left(\begin{smallmatrix} 1 \\ 0 \\ 1 \\ 1  \end{smallmatrix}\right)$ & $\left(\begin{smallmatrix} 1 \\ 1 \\ 0 \\ 1 \end{smallmatrix}\right)$ & $\left(\begin{smallmatrix} -1/2 \\ 0 \\ 0 \\ 0 \end{smallmatrix}\right)$ & $\left(\begin{smallmatrix} -1 \\ -1/2 \\ -1/2 \\ -1 \end{smallmatrix}\right)$ \\
         \padedvphantom{I}{1ex}{1ex}
        $\mathcal{L}_1(C)$ & {\tiny $(6b^2-1)/{24}$ }& {\tiny $7/48$ }  & {\tiny $7/48$ } &  {\tiny $5/24$ }& {\tiny $5/24$} & {\tiny $-1/48$}  & {\tiny $1/48$} \\
    \hline 
    \padedvphantom{I}{3.5ex}{3.5ex}
    $\mathcal{DL}_1(A)$ & \multicolumn{7}{c|}{$\left(\begin{smallmatrix}   1 & 0 & 0 & -1/2\\  0 & 1 & 0 & -1/2 \\ 0 & 0 & 1 & -1/2 \\ -1/2 & -1/2 & -1/2 & 1  \end{smallmatrix}\right)$} \\
   \hline 
   \padedvphantom{I}{3ex}{3ex}
        $\mathcal{DL}_1(B)$ & $\left(\begin{smallmatrix} 0 \\ b \\ -b \\ 0 \end{smallmatrix}\right)$ & $\left(\begin{smallmatrix}  0 \\  -1/2 \\ 1/2 \\ 1/4  \end{smallmatrix}\right)$ & $\left(\begin{smallmatrix} 0 \\ 1/2 \\ -1/2 \\ 1/4  \end{smallmatrix}\right)$ & $\left(\begin{smallmatrix} 1/2 \\ -1/2 \\ 1/2 \\ 0 \end{smallmatrix}\right)$ & $\left(\begin{smallmatrix} 1/2 \\ 1/2 \\ -1/2 \\ 0 \end{smallmatrix}\right)$ & $\left(\begin{smallmatrix}
            -1/2 \\ 0 \\ 0 \\ 1/4
        \end{smallmatrix}\right)$  & $\left(\begin{smallmatrix}
            -1/2 \\ 0 \\ 0 \\ 0
        \end{smallmatrix}\right)$\\
         \padedvphantom{I}{1ex}{1ex}
        $\mathcal{DL}_1(C)$ & {\tiny $(3b^2-2)/{16}$ } & {\tiny $1/16$}   & {\tiny $1/16$} & {\tiny $1/8$} & {\tiny $1/8$}  &{\tiny $-1/48$} &  {\tiny $1/16$} \\
    \hline
   \end{tabular}
\end{table}
}

Since $i$, $j$ and $k$ are symmetric in the quadratic form generated by $\mathcal{L}_1(A)$ and $\mathcal{DL}_1(A)$, we can get new modular triples by exchanging $i$, $j$ and $k$. Note that the matrices $\mathcal{L}_1(A)$ and $\mathcal{DL}_1(A)$ appeared in \cite{Wang-counterexample} in a different way. The second author \cite[Theorem 1.5]{Wang-counterexample} proved that $f_{\mathcal{DL}_1(A),B_i,1/16}(q)$ is modular for
$B_1=(0,1/2,1/2,-1/2)^\mathrm{T}$ and $B_2=(0,1/2,1/2,0)^\mathrm{T}$ but $f_{\mathcal{L}_1(A),\mathcal{D}(B_i),C'}(q)$ ($i=1,2$) are not modular for any $C'$, and thereby provides the first counterexamples to Zagier's duality expectation.

We now justify the modularity of the lift-dual $\mathcal{DL}_1(A,B,C)$ in Table \ref{Wang-3-Exam7}.
\begin{theorem}
Let
\begin{align}
F(u,v,w,t;q^2):=\sum_{i,j,k,l\geq 0} \frac{u^iv^jw^kt^lq^{i^2+j^2+k^2+l^2-il-jl-kl}}{(q^2;q^2)_i(q^2;q^2)_j(q^2;q^2)_k(q^2;q^2)_l}.
\end{align}
We have
\begin{align}
&F(q^{b},q^{-b},1,1;q^2)=\frac{J_{2}^{3}(-q^{1-b},-q^{1+b};q^2)_{\infty}(-q^{2+b},-q^{2-b};q^4)_{\infty}}{J_{1}^{2}J_{4}} \nonumber\\
&\qquad +2q\frac{J_{4}^{3}(-q^{-b},-q^{2+b};q^2)_{\infty}(-q^{b},-q^{4-b};q^4)_{\infty}}{J_{2}^{3}}, \quad b\in \mathbb{Q},\label{thm1-1} \\
&F(q^{-1},q,q,1;q^2)= 2\frac{J_{4}^{9}}{J_{2}^{7}J_{8}^{2}}+\frac{J_{2}^{7}J_{8}^{2}}{J_{1}^{4}J_{4}^{5}}, \label{thm1-2} \\
&F(q,q^{-1},1,q^{1/2};q^2)=2\frac{J_{4}J_{5,12}}{J_{1}^{2}}+2q^{1/2}\frac{J_{2}^{3}J_{2,12}}{J_{1}^{3}J_{4}}-2q\frac{J_{4}J_{1,12}}{J_{1}^{2}},  \label{thm1-3}\\
&F(1,1,q^{-1},q^{1/2};q^2)=2\frac{J_2^3J_6^2}{J_1^3J_4J_{12}}+4q^{1/2}\frac{J_3J_4J_{12}}{J_1^2J_6},  \label{thm1-4}\\
&F(1,1,q^{-1},1;q^2)=8\frac{J_{2}^{3}}{J_{1}^{3}}. \label{thm1-5}
\end{align}
\end{theorem}
\begin{proof}
Summing over $i,j,k$ using \eqref{Euler1} and then splitting the sum into two parts according to the parity of $l$, we have
\begin{align}\label{3-Exam7-jk-in}
&F(u,v,w,t;q^2)=\sum_{l\geq 0} \frac{t^{l}q^{l^2}}{(q^2;q^2)_l}(-uq^{1-l};q^2)_{\infty}(-vq^{1-l};q^2)_{\infty}(-wq^{1-l};q^2)_{\infty} \nonumber\\
&=(-uq,-vq,-wq;q^2)_{\infty}\sum_{l\geq 0} \frac{t^{2l}q^{4l^2}}{(q^2;q^2)_{2l}}(-uq^{1-2l},-vq^{1-2l},-wq^{1-2l};q^2)_{l} \nonumber\\
&\quad +(-u,-v,-w;q^2)_{\infty}\sum_{l\geq 0} \frac{t^{2l+1}q^{4l^2+4l+1}}{(q^2;q^2)_{2l+1}}(-uq^{-2l},-vq^{-2l},-wq^{-2l};q^2)_{l} \nonumber\\
&=(-uq,-vq,-wq;q^2)_{\infty}\sum_{l\geq 0} \frac{t^{2l}u^lv^lw^lq^{l^2}(-q/u,-q/v,-q/w;q^2)_{l}}{(q^2;q^2)_{2l}} \nonumber\\
&\quad +(-u,-v,-w;q^2)_{\infty}\sum_{l\geq 0} \frac{t^{2l+1}u^lv^lw^lq^{l^2+l+1}(-q^{2}/u,-q^{2}/v,-q^{2}/w;q^2)_{l}}{(q^2;q^2)_{2l+1}}.
\end{align}

(1) By \eqref{3-Exam7-jk-in} we have
\begin{align}
&F(q^{b},q^{-b},1,1;q^2)=(-q,-q^{1+b},-q^{1-b};q^2)_{\infty}\sum_{l\geq 0} \frac{q^{l^2}(-q,-q^{1+b},-q^{1-b};q^2)_{l}}{(q^2;q^2)_{2l}} \nonumber\\
&\quad +(-1,-q^{b},-q^{-b};q^2)_{\infty}\sum_{l\geq 0} \frac{q^{l^2+l+1}(-q^{2},-q^{2+b},-q^{2-b};q^2)_{l}}{(q^2;q^2)_{2l+1}} \nonumber\\
=&(-q,-q^{1+b},-q^{1-b};q^2)_{\infty}{}_2\phi_2\bigg(\genfrac{}{}{0pt}{} {-q^{1+b},-q^{1-b}}{q,-q^2};q^2,-q  \bigg) \nonumber\\
&+\frac{q}{1-q^2}(-1,-q^{-b},-q^{b};q^2)_{\infty}{}_2\phi_2\bigg(\genfrac{}{}{0pt}{} {-q^{2+b},-q^{2-b}}{-q^3,q^3};q^2,-q^2  \bigg).
\end{align}
Substituting \eqref{Bailey's} and  \eqref{Gauss'} into the above identity, we obtain \eqref{thm1-1}.

(2) By \eqref{3-Exam7-jk-in} we have
\begin{align}
&F(q^{-1},q,q,1;q^2)=(-1,-q^{2},-q^{2};q^2)_{\infty}\sum_{l\geq 0} \frac{q^{l^2+l}(-q^{2},-1,-1;q^2)_{l}}{(q^2;q^2)_{2l}} \nonumber\\
&\quad +(-q^{-1},-q,-q;q^2)_{\infty}\sum_{l\geq 0} \frac{q^{l^2+2l+1}(-q^{3},-q,-q;q^2)_{l}}{(q^2;q^2)_{2l+1}} \nonumber\\
&=(-1,-q^{2},-q^{2};q^2)_{\infty}{}_2\phi_2\bigg(\genfrac{}{}{0pt}{} {-1,-1}{-q,q};q^2,-q^2  \bigg) \nonumber\\
&\quad +\frac{1}{1-q}(-q,-q,-q;q^2)_{\infty}{}_2\phi_2\bigg(\genfrac{}{}{0pt}{} {-q,-q}{-q^2,q^3};q^2,-q^3  \bigg).
\end{align}
Substituting \eqref{Gauss'} and \eqref{Bailey's} into it, we obtain \eqref{thm1-2}.

(3) By \eqref{3-Exam7-jk-in} we have
\begin{align}
&F(q,q^{-1},1,q^{1/2};q^2)=(-q,-q^{2},-1;q^2)_{\infty}\sum_{l\geq 0} \frac{q^{l^2+l}(-1,-q^{2},-q;q^2)_{l}}{(q^2;q^2)_{2l}} \nonumber\\
&\quad +(-q,-q^{-1},-1;q^2)_{\infty}\sum_{l\geq 0} \frac{q^{l^2+2l+\frac{3}{2}}(-q,-q^{3},-q^2;q^2)_{l}}{(q^2;q^2)_{2l+1}} \nonumber\\
&=2(-q,-q^{2},-q^{2};q^2)_{\infty}\sum_{l\geq 0} \frac{q^{l^2+l}(-1;q^2)_{l}}{(q;q)_{2l}}   \nonumber\\
&\quad +2(-q,-q,-q^2;q^2)_{\infty}\sum_{l\geq 0} \frac{q^{l^2+2l+\frac{1}{2}}(-q;q^2)_{l}}{(q;q)_{2l+1}}.
\end{align}
Substituting \eqref{S. 48} and \eqref{S. 50} into it, we obtain \eqref{thm1-3}.

(4) By \eqref{3-Exam7-jk-in} we have
\begin{align}
&F(1,1,q^{-1},q^{1/2};q^2)=(-q,-q,-1;q^2)_{\infty}\sum_{l\geq 0} \frac{q^{l^2}(-q,-q,-q^{2};q^2)_{l}}{(q^2;q^2)_{2l}}  \\
&\quad +(-1,-1,-q^{-1};q^2)_{\infty}\sum_{l\geq 0} \frac{q^{l^2+l+\frac{3}{2}}(-q^2,-q^2,-q^{3};q^2)_{l}}{(q^2;q^2)_{2l+1}} \nonumber\\
&=2(-q,-q,-q^2;q^2)_{\infty}\sum_{l\geq 0} \frac{q^{l^2}(-q;q^2)_{l}}{(q;q)_{2l}} +4(-q,-q^2,-q^2;q^2)_{\infty}\sum_{l\geq 0} \frac{q^{l^2+l+\frac{1}{2}}(-q^2;q^2)_{l}}{(q;q)_{2l+1}}. \nonumber
\end{align}
Substituting \eqref{S. 29} and \eqref{S. 28} into it, we obtain \eqref{thm1-4}.

(5) By \eqref{3-Exam7-jk-in} we have
\begin{align}
&F(1,1,q^{-1},1;q^2)=(-q,-q,-1;q^2)_{\infty}\sum_{l\geq 0} \frac{q^{l^2-l}(-q,-q,-q^{2};q^2)_{l}}{(q^2;q^2)_{2l}} \nonumber\\
&\quad +(-1,-1,-q^{-1};q^2)_{\infty}\sum_{l\geq 0} \frac{q^{l^2+1}(-q^2,-q^2,-q^{3};q^2)_{l}}{(q^2;q^2)_{2l+1}} \nonumber\\
&=2(-q,-q,-q^2;q^2)_{\infty}\sum_{l\geq 0} \frac{q^{l^2-l}(-q;q^2)_{l}}{(q;q)_{2l}} \nonumber\\
&\quad +4(-q,-q^2,-q^2;q^2)_{\infty}\sum_{l\geq 0} \frac{q^{l^2}(-q^2;q^2)_{l}}{(q;q)_{2l+1}} \nonumber\\
&=2(-q,-q,-q^2;q^2)_{\infty}\lim\limits_{a\to 0}\sum_{l\geq 0} \frac{(1/a,-q;q^2)_{l}(-a)^{l}}{(q,q^2;q^2)_{l}} \nonumber\\
&\quad +\frac{4}{1-q}(-q,-q^2,-q^2;q^2)_{\infty}\lim\limits_{a\to 0}\sum_{l\geq 0} \frac{(1/a,-q^2;q^2)_{l}(-aq)^{l}}{(q^3,q^2;q^2)_{l}} \nonumber\\
&=2(-q,-q,-q^2;q^2)_{\infty}\lim\limits_{a\to 0}\frac{(aq,-1;q^2)_{\infty}}{(q,-a;q^2)_{\infty}}  \nonumber\\
&\quad +\frac{4}{1-q}(-q,-q^2,-q^2;q^2)_{\infty}\lim\limits_{a\to 0}\frac{(aq^3,-q;q^2)_{\infty}}{(q^3,-aq;q^2)_{\infty}} \quad(\text{by \eqref{Gauss}})  \nonumber\\
&=8\frac{(-q,-q,-q^2,-q^2;q^2)_{\infty}}{(q;q^2)_{\infty}}.
\end{align}
This proves \eqref{thm1-5}.
\end{proof}

\subsection{Lift-dual of Zagier's Example 8}
We list the rank three modular triples $(A,B,C)$ from Zagier's Example 8 (see \cite[Table 3]{Zagier} or \cite[Example 8]{Wang-rank3}) in Table \ref{tab-A-exam1}. We label the vectors $B$ and scalars $C$ as $B_i$ and $C_i$ ($i=1,2,\dots, 7$) according to their order of appearance.
From Table \ref{tab-lift} we see that only the lift via $\mathcal{L}_1$ is meaningful, and we checked that $\mathcal{L}_1(A)$ is indeed positive definite.  We record the lift and lift-dual via $\mathcal{L}_1$ in Table \ref{tab-A-exam1} as well.
{\small
\begin{table}[htbp]
\centering
 \caption{Lift and lift-dual via $\mathcal{L}_1$ of the modular triples $(A,B,C)$ in Zagier's Example 8}
    \label{tab-A-exam1}
    \begin{tabular}{|c|ccccccc|}
\hline 
\padedvphantom{I}{2ex}{2ex}
$A$ & \multicolumn{7}{c|}{$\left(\begin{smallmatrix}
            4 & 2 & 1 \\ 2 & 2 & 0 \\ 1 & 0 & 1
        \end{smallmatrix}\right)$} \\
\hline 
\padedvphantom{I}{2ex}{2ex}
$B$ &  $\left(\begin{smallmatrix}
            -1 \\ -1 \\ 1/2
        \end{smallmatrix}\right)$ & $\left(\begin{smallmatrix}
            0 \\ -1/2 \\ 1/2
        \end{smallmatrix}\right)$ & $\left(\begin{smallmatrix}
            0 \\ 0 \\ 1/2
        \end{smallmatrix}\right)$  &  $\left(\begin{smallmatrix}
            2 \\ 1 \\ 1/2
        \end{smallmatrix}\right)$ & $\left(\begin{smallmatrix}
            0 \\ 0 \\ 0
        \end{smallmatrix}\right)$ &  $\left(\begin{smallmatrix}
            2 \\ 1 \\ 1
        \end{smallmatrix}\right)$ & $\left(\begin{smallmatrix}
            1 \\ 0 \\ 1/2
        \end{smallmatrix}\right)$ \\
         \padedvphantom{I}{1ex}{1ex}
$C$ & {\tiny $1/12$ } &{\tiny $1/48$} & {\tiny $0$} &  {\tiny $1/3$} & {\tiny $-1/24$ } & {\tiny $11/24$}  & {\tiny $1/12$ }\\
\hline 
\padedvphantom{I}{3ex}{3ex}
   $\mathcal{L}_1(A)$
    & \multicolumn{7}{c|}{$\left(\begin{smallmatrix}  4 & 2 & 1 & 3 \\ 2 & 2 & 1 & 2 \\ 1 & 1 & 1 & 1 \\ 3 & 2 & 1 & 3 \end{smallmatrix}\right)$}  \\
\hline
\padedvphantom{I}{3ex}{3ex}
   $\mathcal{L}_1(B)$   & $\left(\begin{smallmatrix} -1 \\ -1 \\ 1/2 \\ -1/2 \end{smallmatrix}\right)$ & $\left(\begin{smallmatrix} 0 \\ -1/2 \\ 1/2 \\ 0 \end{smallmatrix}\right)$ &  $\left(\begin{smallmatrix} 0 \\ 0 \\ 1/2 \\ 1/2 \end{smallmatrix}\right)$ & $\left(\begin{smallmatrix} 2 \\ 1 \\ 1/2 \\ 3/2 \end{smallmatrix}\right)$
    & $\left(\begin{smallmatrix} 0 \\ 0 \\ 0 \\ 0  \end{smallmatrix}\right)$ & $\left(\begin{smallmatrix} 2 \\ 1 \\ 1 \\ 2 \end{smallmatrix}\right)$ & $\left(\begin{smallmatrix} 1 \\ 0 \\ 1/2 \\ 1/2 \end{smallmatrix}\right)$  \\
 \padedvphantom{I}{1ex}{1ex}
 $\mathcal{L}_1(C)$    & {\tiny $1/12$}  & {\tiny $1/48$ } & {\tiny $0$} &{\tiny $1/3$} & {\tiny $-1/24$} & {\tiny $11/24$ } & {\tiny $1/12$} \\
 \hline 
 \padedvphantom{I}{3ex}{3ex}
   $\mathcal{D}\mathcal{L}_1(A)$
    & \multicolumn{7}{c|}{$\left(\begin{smallmatrix}  1 & 0 & 0 & -1 \\ 0 & 2 & -1 & -1 \\ 0 & -1 & 2 & 0 \\ -1 & -1 & 0 & 2 \end{smallmatrix}\right)$}  \\
 \hline
 \padedvphantom{I}{3ex}{3ex}
   $\mathcal{D}\mathcal{L}_1(B)$   & $\left(\begin{smallmatrix} -1/2 \\ -2 \\ 2 \\ 1 \end{smallmatrix}\right)$ & $\left(\begin{smallmatrix} 0 \\ -3/2 \\ 3/2 \\ 1/2 \end{smallmatrix}\right)$ &  $\left(\begin{smallmatrix} -1/2 \\ -1 \\ 1 \\ 1 \end{smallmatrix}\right)$ & $\left(\begin{smallmatrix} 1/2 \\ 0 \\ 0 \\ 0 \end{smallmatrix}\right)$
    & $\left(\begin{smallmatrix} 0 \\ 0 \\ 0 \\ 0  \end{smallmatrix}\right)$ & $\left(\begin{smallmatrix} 0 \\ -1 \\ 1 \\ 1 \end{smallmatrix}\right)$ & $\left(\begin{smallmatrix} 1/2 \\ -1 \\ 1 \\ 0 \end{smallmatrix}\right)$  \\
 \padedvphantom{I}{1ex}{1ex}
 $\mathcal{D}\mathcal{L}_1(C)$    & {\tiny $5/4$ } & {\tiny $9/16$ } & {\tiny $1/3$} & {\tiny $0$} & {\tiny $-1/8$} & {\tiny $3/8$} & {\tiny $1/4$} \\
    \hline
    \end{tabular}
\end{table}
}

The lift-dual matrix $\mathcal{DL}_1(A)$ is essentially the rank four tadpole Cartan matrix (with some rows and columns interchanged) discussed in \cite{CMP,MW24,Shi-Wang}. The modularity for the Nahm sums with vectors $\mathcal{DL}_1(B_4)=(1/2,0,0,0)^\mathrm{T}$, $\mathcal{D}\mathcal{L}_1(B_5)=(0,0,0,0)^\mathrm{T}$ and $\mathcal{D}\mathcal{L}_1(B_6)=(0,-1,1,1)^\mathrm{T}$ has been confirmed by Shi and Wang  \cite[Theorem 1.5]{Shi-Wang}. A conjectural Rogers--Ramanujan type identity for the case $\mathcal{D}\mathcal{L}_1(B_3)=(-1/2,-1,1,1)^\mathrm{T}$  was presented in \cite[(1.24)]{Shi-Wang}, which can be stated equivalently as
\begin{align}\label{eq-conj-Shi-Wang-equivalent}
\sum_{i,j,k,l\geq 0} \frac{q^{\frac{1}{2}(i^2-i)+j^2+k^2+l^2-il-jk-jl-j+k+l}}{(q;q)_i(q;q)_j(q;q)_k(q;q)_l}=6\frac{J_2^2J_3^3}{J_1^5}.
\end{align}
The other vectors are new, but we are not able to confirm their modularity at this stage. Through Maple we find the following formulas for the Nahm sums associated with the other triples $\mathcal{DL}_1(A,B_i,C_i)$ ($i=1,2,7$).
\begin{conj}\label{conj-new-tadpole}
We have
\begin{align}
&\sum_{i,j,k,l\geq 0} \frac{q^{\frac{1}{2}(i^2-i)+j^2+k^2+l^2-il-jk-jl-2j+2k+l}}{(q;q)_i(q;q)_j(q;q)_k(q;q)_l}=8q^{-1}\frac{J_2^6}{J_1^6}, \label{conj-eq-1} \\
&\sum_{i,j,k,l\geq 0} \frac{q^{i^2+2j^2+2k^2+2l^2-2il-2jk-2jl-3j+3k+l}}{(q^2;q^2)_i(q^2;q^2)_j(q^2;q^2)_k(q^2;q^2)_l}=2q^{-1}\frac{J_2^3}{J_1^3}, \label{conj-eq-2} \\
&\sum_{i,j,k,l\geq 0} \frac{q^{\frac{1}{2}(i^2+i)+j^2+k^2+l^2-il-jk-jl-j+k}}{(q;q)_i(q;q)_j(q;q)_k(q;q)_l}=4\frac{J_2^6}{J_1^6}. \label{conj-eq-3}
\end{align}
\end{conj}
Shi and Wang \cite{Shi-Wang} also confirmed the modularity for the Nahm sum $f_{\mathcal{DL}_1(A),B_8',5/24}(q)$ with vector $B_8'=(0,-1,0,1)^\mathrm{T}$ which can be stated equivalently as \cite[(1.21)]{Shi-Wang}
\begin{align}
\sum_{i,j,k,l\geq 0} \frac{q^{i^2+2j^2+2k^2+2l^2-2il-2jk-2jl-2j+2l}}{(q^2;q^2)_i(q^2;q^2)_j(q^2;q^2)_k(q^2;q^2)_l}=3\frac{J_2J_6^3}{J_1^2J_4^2}. \label{Shi-Wang-eq121}
\end{align}
This additional vector is not obtained by the lift-dual operation. As can be seen from \eqref{eq-conj-Shi-Wang-equivalent}--\eqref{Shi-Wang-eq121}, each of the Nahm sums $f_{\mathcal{DL}_1(A,B_i,C_i)}(q)$ ($i=1,2,3,7$) and  $f_{\mathcal{DL}_1(A),B_8',5/24}(q)$ can be expressed by a single product. In contrast, the Nahm sums  $f_{\mathcal{DL}_1(A,B_i,C_i)}(q)$ ($i=4,5,6$) cannot be expressed as single infinite products. Instead, they can be expressed as the sum of two infinite products (see \cite[Theorem 1.5]{Shi-Wang}).

\subsection{Lift-dual of Zagier's Example 9}
Recall the rank three modular triples from Zagier's Example 9 (see \cite[Table 3]{Zagier} or \cite[Example 9]{Wang-rank3}):
\begin{equation}\label{eq-Zagier-exam9}
    \begin{split}
      &  A=\begin{pmatrix} 6 & 4 & 2 \\ 4 & 4 & 2 \\ 2 & 2 & 2 \end{pmatrix}, \quad B_1=\begin{pmatrix} 0 \\ 0 \\ 0 \end{pmatrix}, B_2=\begin{pmatrix} 1 \\ 0 \\ 0 \end{pmatrix}, B_3=\begin{pmatrix} 2 \\ 1 \\ 0 \end{pmatrix}, B_4=\begin{pmatrix} 3 \\ 2 \\ 1 \end{pmatrix}, \\
       & C_1= -\frac{1}{36}, \quad C_2=\frac{1}{12}, \quad C_3=\frac{11}{36}, \quad C_4=\frac{23}{36}.
    \end{split}
\end{equation}
From Table \ref{tab-lift} we see that only the lift via $\mathcal{L}_2$ is meaningful, and we checked that $\mathcal{L}_2(A)$ is indeed positive definite.  We record the lift and lift-dual via $\mathcal{L}_2$ in Table \ref{tab-A-exam2}.

{\small
\begin{table}[htbp]
\centering
 \caption{Lift and lift-dual via $\mathcal{L}_2$  of the modular triples $(A,B,C)$ in \eqref{eq-Zagier-exam9}}
    \label{tab-A-exam2}
 \begin{tabular}{|c|cccc|c|cccc|}
\hline
        \multicolumn{5}{|c|}{{Lift via $\mathcal{L}_2$}} &   \multicolumn{5}{c|}{{Lift-dual via $\mathcal{L}_2$}} \\
\hline 
\padedvphantom{I}{3ex}{3ex}
   $\mathcal{L}_2(A)$
    & \multicolumn{4}{c|}{$\left(\begin{smallmatrix}  6 & 4 & 3 & 8 \\ 4 & 4 & 2 & 6 \\ 3 & 2 & 2 & 4 \\ 8 & 6 & 4 & 12 \end{smallmatrix}\right)$}  & $\mathcal{D}\mathcal{L}_2(A)$
    & \multicolumn{4}{c|}{$\left(\begin{smallmatrix} 2 & 0 & -1 & -1 \\ 0 & 1 & 0 & -1/2 \\ -1 & 0 & 2 & 0 \\ -1 & -1/2 & 0 & 1 \end{smallmatrix}\right)$} \\
\hline
\padedvphantom{I}{3ex}{3ex}
   $\mathcal{L}_2(B)$   & $\left(\begin{smallmatrix} 0 \\ 0 \\ 0 \\ 0 \end{smallmatrix}\right)$ & $\left(\begin{smallmatrix} 1 \\ 0 \\ 0 \\ 1 \end{smallmatrix}\right)$ &  $\left(\begin{smallmatrix} 2 \\ 1 \\ 0 \\ 2 \end{smallmatrix}\right)$ & $\left(\begin{smallmatrix} 3 \\ 2 \\ 1 \\ 4 \end{smallmatrix}\right)$
 & $\mathcal{D}\mathcal{L}_2(B)$    & $\left(\begin{smallmatrix} 0 \\ 0 \\ 0 \\ 0  \end{smallmatrix}\right)$ & $\left(\begin{smallmatrix} 1 \\ -1/2 \\ -1 \\ 0 \end{smallmatrix}\right)$ & $\left(\begin{smallmatrix} 2 \\ 0 \\ -2 \\ -1/2 \end{smallmatrix}\right)$ & $\left(\begin{smallmatrix} 1 \\0 \\ -1 \\ 0 \end{smallmatrix}\right)$ \\
 \padedvphantom{I}{1ex}{1ex}
 $\mathcal{L}_2(C)$      & {\tiny $-{1}/{36}$} & {\tiny ${1}/{12}$} & {\tiny ${11}/{36}$} & {\tiny ${23}/{36}$} & $\mathcal{D}\mathcal{L}_2(C)$
    & {\tiny $-{5}/{36}$ } & {\tiny ${1}/{4}$} &  {\tiny ${37}/{36}$ } & {\tiny ${7}/{36}$} \\
    \hline
\end{tabular}
\end{table}
}

Though we cannot prove the modularity at this stage, we find the following identity which verifies the modularity for the second triple $\mathcal{DL}_2(A,B_2,C_2)$. It seems that at least two infinite products will be involved to express Nahm sums associated with the triples $\mathcal{DL}_2(A,B_i,C_i)$ ($i=1, 3, 4$).
\begin{conj}
We have
\begin{align}
\sum_{i,j,k,l\geq 0} \frac{q^{2i^2+j^2+2k^2+l^2-2ik-2il-jl+2i-j-2k}}{(q^2;q^2)_i(q^2;q^2)_j(q^2;q^2)_k(q^2;q^2)_l}=9\frac{J_3^3J_6^2}{J_1J_2^4}. \label{eq-conj-Z9-2}
\end{align}
\end{conj}

\subsection{Lift-dual of Zagier's Example 11}
Recall the modular triples from Zagier's Example 11 (see \cite[Table 3]{Zagier} or \cite[Example 11]{Wang-rank3}):
\begin{align}
\label{eq-Zagier-exam11}
   & A=\begin{pmatrix} 4 & 2 & -1 \\ 2 & 2 & -1 \\ -1 & -1 & 1 \end{pmatrix}, ~~B_1=\begin{pmatrix} 0 \\ 0 \\ 0 \end{pmatrix}, ~~B_2=\begin{pmatrix} 0 \\ 0 \\ 1/2 \end{pmatrix}, ~~B_3=\begin{pmatrix} 1 \\ 0 \\ 0 \end{pmatrix}, ~~
    B_4=\begin{pmatrix} 2 \\ 1 \\ -1/2 \end{pmatrix}, \nonumber \\ &B_5=\begin{pmatrix} 2 \\ 1 \\ 0 \end{pmatrix}, ~~
    C_1= -\frac{1}{16}, ~~ C_2=\frac{1}{24},~~ C_3=\frac{5}{48}, ~~ C_4=\frac{3}{8},~~ C_5=\frac{7}{16}.
\end{align}
From Table \ref{tab-lift} we know that only the lift of $\mathcal{L}_1$ is meaningful, and we checked that $\mathcal{L}_1(A)$ is indeed positive definite.  We record the lift and lift-dual via $\mathcal{L}_1$ in Table \ref{tab-A-exam3}.

{\small
\begin{table}[htbp]
\centering
 \caption{Lift and lift-dual via $\mathcal{L}_1$ of the modular triples $(A,B,C)$ in \eqref{eq-Zagier-exam11}}
    \label{tab-A-exam3}
 \begin{tabular}{|c|ccccc|}
\hline 
\padedvphantom{I}{3ex}{3ex}
   $\mathcal{L}_1(A)$
    &  \multicolumn{5}{c|}{$\left(\begin{smallmatrix}  4 & 2 & -1 & 1 \\ 2 & 2 & 0 & 1 \\ -1 & 0 & 1 & 0 \\ 1 & 1 & 0 & 1 \end{smallmatrix}\right)$} \\
\hline
\padedvphantom{I}{3ex}{3ex}
   $\mathcal{L}_1(B)$   & $\left(\begin{smallmatrix} 0 \\ 0 \\ 0 \\ 0 \end{smallmatrix}\right)$ & $\left(\begin{smallmatrix} 0 \\ 0 \\ 1/2 \\ 1/2 \end{smallmatrix}\right)$ & $\left(\begin{smallmatrix} 1 \\ 0 \\ 0 \\ 0 \end{smallmatrix}\right)$ & $\left(\begin{smallmatrix} 2 \\ 1 \\ -1/2 \\ 1/2 \end{smallmatrix}\right)$ & $\left(\begin{smallmatrix} 2 \\ 1 \\ 0 \\ 1 \end{smallmatrix}\right)$ \\
 \padedvphantom{I}{1ex}{1ex}
 $\mathcal{L}_1(C)$   & {\tiny $-{1}/{16}$} & {\tiny ${1}/{24}$} &  {\tiny ${5}/{48}$} & {\tiny ${3}/{8}$} & {\tiny ${7}/{16}$ }  \\
 \hline 
 \padedvphantom{I}{3ex}{3ex}
    $\mathcal{D}\mathcal{L}_1(A)$
    & \multicolumn{5}{c|}{$\left(\begin{smallmatrix} 1 & -1 & 1 & 0 \\ -1 & 2 & -1 & -1 \\ 1 & -1 & 2 & 0 \\ 0 & -1 & 0 & 2 \end{smallmatrix}\right)$} \\
\hline
\padedvphantom{I}{2ex}{2ex}
   $\mathcal{D}\mathcal{L}_1(B)$    & $\left(\begin{smallmatrix} 0 \\ 0 \\ 0 \\ 0  \end{smallmatrix}\right)$ &$\left(\begin{smallmatrix} 1/2 \\ -1 \\ 1 \\ 1 \end{smallmatrix}\right)$ & $\left(\begin{smallmatrix} 1 \\ -1 \\ 1 \\ 0 \end{smallmatrix}\right)$ &   $\left(\begin{smallmatrix} 1/2 \\0 \\ 0 \\ 0 \end{smallmatrix}\right)$ & $\left(\begin{smallmatrix} 1 \\ -1 \\ 1 \\ 1 \end{smallmatrix}\right)$   \\
 \padedvphantom{I}{1ex}{1ex}
 $\mathcal{D}\mathcal{L}_1(C)$    & {\tiny $-{5}/{48}$} & {\tiny ${7}/{24}$ } & {\tiny ${11}/{48}$} & {\tiny $-{1}/{24}$ } & {\tiny ${19}/{48}$} \\
    \hline
\end{tabular}
\end{table}
}

We find single-product expressions for the Nahm sum associated with the second and third triples $\mathcal{DL}_1(A,B_i,C_i)$ ($i=2,3$). It seems that more products are needed to express Nahm sums associated with the other triples $\mathcal{DL}_1(A,B_i,C_i)$ ($i=1,4,5$).
\begin{conj}
We have
\begin{align}
&\sum_{i,j,k,l\geq 0} \frac{q^{\frac{1}{2}(i^2+i)+j^2+k^2+l^2-ij+ik-jk-jl-j+k+l}}{(q;q)_i(q;q)_j(q;q)_k(q;q)_l}=3\frac{J_2J_3^3}{J_1^4}, \label{eq-conj-Z11-2} \\
&\sum_{i,j,k,l\geq 0} \frac{q^{i^2+2j^2+2k^2+2l^2-2ij+2ik-2jk-2jl+2i-2j+2k}}{(q^2;q^2)_i(q^2;q^2)_j(q^2;q^2)_k(q^2;q^2)_l}=3\frac{J_6^3}{J_1J_2J_4}. \label{eq-conj-Z11-3}
\end{align}
\end{conj}

\subsection{Lift-dual of Zagier's Example 12}
Recall the modular triples from Zagier's Example 12 (see \cite[Table 3]{Zagier} or \cite[Example 12]{Wang-rank3}):
\begin{equation}\label{eq-Zagier-exam12}
    \begin{split}
        A=\begin{pmatrix} 8 & 4 & 1 \\ 4 & 3 & 0 \\ 1 & 0 & 1 \end{pmatrix},~~ B_1=\begin{pmatrix} 0 \\ -1/2 \\ 1/2 \end{pmatrix}, ~~ B_2=\begin{pmatrix} 2 \\ 1/2 \\ 1/2 \end{pmatrix}, ~~ C_1=\frac{1}{40},~~  C_2=\frac{9}{40}.
    \end{split}
\end{equation}
From Table \ref{tab-lift} we see that only the lift via $\mathcal{L}_1$ is meaningful, and we checked that $\mathcal{L}_1(A)$ is indeed positive definite.  We record the lift and lift-dual of these modular triples via $\mathcal{L}_1$ in Table \ref{tab-A-exam4}. We do not find any single-product expressions for the Nahm sums associated with the triples $\mathcal{DL}_1(A,B_i,C_i)$ ($i=1,2$).

{\small
\begin{table}[H]
\centering
\caption{Lift and lift-dual via $\mathcal{L}_1$ of the modular triples $(A,B,C)$ in \eqref{eq-Zagier-exam12}}
    \label{tab-A-exam4}
\begin{tabular}{|c|cc|c|cc|}
\hline
         \multicolumn{3}{|c|}{{Lift via $\mathcal{L}_1$}} & \multicolumn{3}{c|}{{Lift-dual via $\mathcal{L}_1$}}\\
\hline 
\padedvphantom{I}{3.5ex}{3.5ex}
    $\mathcal{L}_1(A)$   & \multicolumn{2}{c|}{$\left(\begin{smallmatrix}  8 & 4 & 1 & 5 \\ 4 & 3 & 1 & 3 \\ 1 & 1 & 1 & 1 \\ 5 & 3 & 1 & 4 \end{smallmatrix}\right)$}  & $\mathcal{DL}_1(A)$
    & \multicolumn{2}{c|}{$\left(\begin{smallmatrix} 2/3 & -1/3 & 1/3 & -2/3 \\ -1/3 & 5/3 & -2/3 & -2/3 \\ 1/3 & -2/3 & 5/3 & -1/3 \\ -2/3 & -2/3 & -1/3 & 5/3 \end{smallmatrix}\right)$} \\
\hline 
\padedvphantom{I}{3.5ex}{3.5ex}
   $\mathcal{L}_1(B)$
     & $\left(\begin{smallmatrix} 0 \\ -1/2 \\ 1/2 \\ 0 \end{smallmatrix}\right)$ & $\left(\begin{smallmatrix} 2 \\ 1/2 \\ 1/2 \\ 1 \end{smallmatrix}\right)$
     &$\mathcal{DL}_1(B)$
    & $\left(\begin{smallmatrix} 1/3 \\ -7/6 \\ 7/6 \\ 1/6  \end{smallmatrix}\right)$ & $\left(\begin{smallmatrix} 2/3 \\ -5/6 \\ 5/6 \\ -1/6 \end{smallmatrix}\right)$ \\
     \padedvphantom{I}{1ex}{1ex}
   $\mathcal{L}_1(C)$
    & {\tiny ${1}/{40}$} & {\tiny ${9}/{40}$} & $\mathcal{DL}_1(C)$
    & {\tiny ${47}/{120}$}  & {\tiny ${23}/{120}$} \\
    \hline
\end{tabular}
\end{table}
}

\section{Lift-dual of the CW Examples 1--3}\label{sec-CW}

\subsection{Lift-dual of the CW Example 1}
We list the modular triples in \cite[(6.31) and Table 3]{Cao-Wang2024} and their lift and lift-dual in Table \ref{tab-4}.
{\small
\begin{table}[htbp]
\centering
    \caption{Modular triples from \cite[Table 3]{Cao-Wang2024} and their lift and lift-dual via $\mathcal{L}_2$}\label{tab-4}
    \label{CW-thm6.2}
\begin{tabular}{|c|cccccc|}
    \hline 
    \padedvphantom{I}{3ex}{3ex}
    $A$ & \multicolumn{6}{c|}{$\left(\begin{smallmatrix} 1 & 0 & -1/2 \\  0 & 1 & -1/2 \\ -1/2 & -1/2 & 1 \end{smallmatrix}\right)$} \\
    \hline 
    \padedvphantom{I}{2.5ex}{2.5ex}
        $B$ &  $\left(\begin{smallmatrix}  -b \\ b \\ 0 \end{smallmatrix}\right)$  & $\left(\begin{smallmatrix}  0 \\ 0 \\ -1/2 \end{smallmatrix}\right)$ &  $\left(\begin{smallmatrix}  -1/4 \\ 1/4 \\ 1/2 \end{smallmatrix}\right)$ & $\left(\begin{smallmatrix}  1/4 \\ -1/4 \\ 1/2 \end{smallmatrix}\right)$  & $\left(\begin{smallmatrix}  0 \\ -1/2 \\ 0 \end{smallmatrix}\right)$  & $\left(\begin{smallmatrix}  -1/2 \\ 0 \\ 0 \end{smallmatrix}\right)$ \\
        \padedvphantom{I}{1ex}{1ex}
        $C$ & {\tiny $(8b^2-1)/12$ }  & {\tiny $1/12$ }&  {\tiny $1/24$ } & {\tiny  $1/24$ } & {\tiny  $1/{24}$ } & {\tiny  $1/24$ }     \\
        \hline 
        \padedvphantom{I}{3ex}{3ex}
        $B$   &  $\left(\begin{smallmatrix}  0 \\ -1/2 \\ 1/4 \end{smallmatrix}\right)$  &  $\left(\begin{smallmatrix}  -1/2 \\ 0 \\ 1/4 \end{smallmatrix}\right)$   & $\left(\begin{smallmatrix}  0 \\ -1/2 \\ 1/2 \end{smallmatrix}\right)$  &  $\left(\begin{smallmatrix}  -1/2 \\ 0 \\ 1/2 \end{smallmatrix}\right)$   &  $\left(\begin{smallmatrix}  1/2\\ 1/2 \\ -1/2 \end{smallmatrix}\right)$ & $\left(\begin{smallmatrix}  1/2\\ 1/2 \\ 0  \end{smallmatrix}\right)$  \\
         \padedvphantom{I}{1ex}{1ex}
        $C$  & {\tiny  $1/96$} &  {\tiny $1/96$}  & {\tiny $1/24$}     & {\tiny $1/24$}   & {\tiny  $1/12$}  &{\tiny  $1/12$}   \\
        \hline 
        \padedvphantom{I}{4ex}{4ex}
         $\mathcal{L}_2(A)$ &\multicolumn{6}{c|}{$\left(\begin{smallmatrix}  1 & 0 & 1/2 & 1/2 \\  0 & 1 & -1/2 & -1/2 \\   1/2 & -1/2 & 1 & 1/2 \\ 1/2 & -1/2  & 1/2 & 1 \end{smallmatrix}\right)$} \\
    \hline 
    \padedvphantom{I}{3.5ex}{3.5ex}
    $\mathcal{L}_2(B)$ &  $\left(\begin{smallmatrix}  -b \\ b \\ 0 \\ -b \end{smallmatrix}\right)$  & $\left(\begin{smallmatrix}  0 \\ 0 \\ -1/2 \\ -1/2 \end{smallmatrix}\right)$ &  $\left(\begin{smallmatrix}  -1/4 \\ 1/4 \\ 1/2 \\ 1/4 \end{smallmatrix}\right)$ & $\left(\begin{smallmatrix}  1/4 \\ -1/4 \\ 1/2 \\ 3/4 \end{smallmatrix}\right)$  & $\left(\begin{smallmatrix}  0 \\ -1/2 \\ 0 \\ 0 \end{smallmatrix}\right)$  & $\left(\begin{smallmatrix}  -1/2 \\ 0 \\ 0 \\ -1/2 \end{smallmatrix}\right)$ \\
     \padedvphantom{I}{1ex}{1ex}
    $\mathcal{L}_2(C)$ & {\tiny $(8b^2-1)/12$}   & {\tiny $1/12$ } & {\tiny $1/24$} & {\tiny  $1/24$ } & {\tiny  $1/{24}$ } &  {\tiny $1/24$ }      \\
    \hline 
    \padedvphantom{I}{3.5ex}{3.5ex}
    $\mathcal{L}_2(B)$&  $\left(\begin{smallmatrix}  0 \\ -1/2 \\ 1/4 \\ 1/4 \end{smallmatrix}\right)$  &  $\left(\begin{smallmatrix}  -1/2 \\ 0 \\ 1/4 \\ -1/4 \end{smallmatrix}\right)$   & $\left(\begin{smallmatrix}  0 \\ -1/2 \\ 1/2 \\ 1/2 \end{smallmatrix}\right)$  &  $\left(\begin{smallmatrix}  -1/2 \\ 0 \\ 1/2 \\ 0 \end{smallmatrix}\right)$   &  $\left(\begin{smallmatrix}  1/2\\ 1/2 \\ -1/2 \\ 0 \end{smallmatrix}\right)$ & $\left(\begin{smallmatrix}  1/2\\ 1/2 \\ 0 \\ 1/2 \end{smallmatrix}\right)$  \\
     \padedvphantom{I}{1ex}{1ex}
    $\mathcal{L}_2(C)$  & {\tiny $1/96$ } & {\tiny $1/96$}  &{\tiny  $1/24$ }     & {\tiny $1/24$}   & {\tiny  $1/12$}  & {\tiny $1/12$}   \\
     \hline 
     \padedvphantom{I}{3ex}{3ex}
    $\mathcal{DL}_2(A)$ &\multicolumn{6}{c|}{$\left(\begin{smallmatrix}  2 & -1 & -1 & -1 \\  -1 & 2 & 1 & 1 \\  -1 & 1 & 2 & 0  \\ -1 & 1  & 0 & 2 \end{smallmatrix}\right)$} \\
    \hline 
    \padedvphantom{I}{3ex}{3ex}
        $\mathcal{DL}_2(B)$ &  $\left(\begin{smallmatrix}  -2b \\ 2b \\ 2b \\ 0 \end{smallmatrix}\right)$  & $\left(\begin{smallmatrix}  1 \\ -1 \\ -1 \\ -1  \end{smallmatrix}\right)$ &  $\left(\begin{smallmatrix}  -3/2 \\ 3/2 \\ 3/2 \\ 1 \end{smallmatrix}\right)$ & $\left(\begin{smallmatrix}  -1/2 \\ 1/2 \\ 1/2 \\ 1  \end{smallmatrix}\right)$  & $\left(\begin{smallmatrix}  1/2 \\ -1 \\ -1/2 \\ -1/2 \end{smallmatrix}\right)$  & $\left(\begin{smallmatrix}  -1/2 \\ 0 \\ 1/2 \\ -1/2  \end{smallmatrix}\right)$ \\
         \padedvphantom{I}{1ex}{1ex}
        $\mathcal{DL}_2(C)$ &  {\tiny $(16b^2-1)/12$}   & {\tiny $1/4$}   & {\tiny $2/3$} &  {\tiny $1/6$ } &  {\tiny $1/24$ } &  {\tiny $1/24$ }     \\
        \hline 
        \padedvphantom{I}{3ex}{3ex}
        $\mathcal{DL}_2(B)$   &  $\left(\begin{smallmatrix}   0 \\ -1/2 \\ 0 \\ 0  \end{smallmatrix}\right)$  &  $\left(\begin{smallmatrix}  -1 \\ 1/2 \\ 1 \\ 0  \end{smallmatrix}\right)$   & $\left(\begin{smallmatrix}  -1/2 \\ 0 \\ 1/2 \\ 1/2  \end{smallmatrix}\right)$  &  $\left(\begin{smallmatrix}   -3/2 \\ 1 \\ 3/2 \\ 1/2 \end{smallmatrix}\right)$   &  $\left(\begin{smallmatrix}   1 \\ 0 \\ -1 \\ 0  \end{smallmatrix}\right)$ & $\left(\begin{smallmatrix}  0 \\ 1 \\ 0 \\ 1  \end{smallmatrix}\right)$  \\
         \padedvphantom{I}{1ex}{1ex}
        $\mathcal{DL}_2(C)$  & {\tiny $-5/96$} &{\tiny  $19/96$ }  &  {\tiny $1/24$}     & {\tiny  $13/24$ }   & {\tiny  $1/4$ }  &{\tiny  $1/4$ }   \\
        \hline
\end{tabular}
\end{table}
}

Applying $\mathcal{L}_1$ or $\mathcal{L}_2$ will lift them to essentially the same set of rank four modular triples (with the first two rows/columns in the matrix interchanged). In contrast, from Table \ref{tab-lift} we see that $\mathcal{L}_3(A)$ is a singular matrix. Hence, we only discuss the lift via the operator $\mathcal{L}_2$.

Note that $n_3$ and $n_4$ are symmetric to each other in the quadratic form generated by $\mathcal{L}_2(A)$ and $\mathcal{DL}_2(A)$. Hence we can get some additional modular triples for free by replacing $\mathcal{L}_2(B)=(b_1,b_2,b_3,b_4)^\mathrm{T}$ in Table \ref{tab-4} with $(b_1,b_2,b_4,b_3)^\mathrm{T}$. Similarly, we can get candidates for free simply by interchanging the third and fourth coordinates in the vector $\mathcal{DL}_2(B)$.

We now confirm the modularity of the candidates $\mathcal{DL}_2(A,B,C)$ in Table \ref{tab-4}.
\begin{theorem}\label{CW-thm6.2-ij3-in}
Let
\begin{align}
F(u,v,w,t;q^2):=\sum_{i,j,k,l\geq 0} \frac{u^iv^jw^kt^lq^{2i^2+2j^2+2k^2+2l^2-2ij-2ik-2il+2jk+2jl}}{(q^2;q^2)_i(q^2;q^2)_j(q^2;q^2)_k(q^2;q^2)_l}.
\end{align}
We have
\begin{align}
F(u,v,w,t;q^2)&=F(u,v,t,w;q^2), \label{thm3-symmetry} \\
   F(q^{-2},q,q^2,1;q^2)&=\frac{1}{J_{1}J_{2}J_{16}}\Big(2\frac{J_{8}^{2}J_{1,8}J_{6,16}}{J_{4}}+\frac{J_{4}^{5}J_{3,8}J_{2,16}}{J_{2}^{2}J_{8}^{2}}\Big), \label{thm3-1} \\
   F(q^{-1},1,q,q^{-1};q^2)&=2\frac{J_{2}^{2}}{J_{1}^{2}},  \label{thm3-2}\\
   F(q^{-3},q^{2},q^{3},q;q^2)&=q^{-1}\frac{J_{2}^{2}}{J_{1}^{2}},  \label{thm3-3}\\
   F(q,q^{-2},q^{-1},q^{-1};q^2)&=2\frac{J_{2}^2}{J_{1}^2},  \label{thm3-4}\\
   F(1,q^{-1},1,1;q^2) &=\frac{J_{4}^5J_{1,8}J_{6,16}}{J_{1}J_{2}^3J_{8}^2J_{16}}+2q\frac{J_{8}^2J_{3,8}J_{2,16}}{J_{1}J_{2}J_{4}J_{16}},  \label{thm3-5}\\
   F(q^{-1},1,q,q;q^2) &=\frac{J_{2}^{2}}{J_{1}^{2}}, \label{thm3-6}\\
   F(q^{-3},q^3,q^3,q^{2};q^2)&=q^{-1}\frac{J_{6}^2}{J_{1}J_{3}},  \label{thm3-7}\\
   F(q^{-1},q,q,q^2;q^2)&=\frac{J_{6}^2}{J_{1}J_{3}},  \label{thm3-8}\\
   F(u,1/u,1/u,1;q^2)&=\frac{1}{J_{2}^{2}}\Big((-q^2u,-q^2/u,q^4;q^4)_{\infty}(-q^6u,-q^6/u,q^{12};q^{12})_{\infty} \nonumber\\
&\quad +q^2(-u,-q^4/u,q^4;q^4)_{\infty}(-q^{12}u,-1/u,q^{12};q^{12})_{\infty}\Big),  \label{thm3-9}\\
   F(1,q^2,1,q^2;q^2)&=\frac{1}{3}\frac{J_{12}^{2}}{J_{2}^{2}J_{24}}\Big(\frac{J_{4}^{5}J_{4,24}}{J_{2}^{2}J_{8}^{2}J_{2,12}}+\frac{2J_{8}^{3}}{J_{4}^{2}} \Big), \label{thm3-10}\\
   F(q^2,1,q^{-2},1;q^2)&=\frac{2}{3}\frac{J_{12}^{2}}{J_{2}^{2}J_{24}}\Big(\frac{J_{4}^{5}J_{4,24}}{J_{2}^{2}J_{8}^{2}J_{2,12}}+\frac{2J_{8}^{3}}{J_{4}^{2}} \Big), \label{thm3-11} \\
    F(q^2,q^{-2},q^{-2},q^{-2};q^2)&=6\frac{J_6^3}{J_2^3}. \label{thm3-12}
\end{align}
\end{theorem}
\begin{proof}
The identity \eqref{thm3-symmetry} is trivial from the definition since $k$ and $l$ are symmetric to each other.

By (\ref{Euler1})--\eqref{Jacobi}, we have
\begin{align}\label{CW-thm6.2-ij3-in-1}
&F(u,v,w,t;q^2)=
\sum_{i,j,k,l\geq 0} \frac{u^iv^jw^kt^lq^{(k-l)^2+(i-j-k-l)^2+i^2+j^2}}{(q^2;q^2)_i(q^2;q^2)_j(q^2;q^2)_k(q^2;q^2)_l}  \\
&=\CT_{x}\CT_{y}\sum_{i\geq 0}\frac{(uy)^iq^{i^2}}{(q^2;q^2)_i}\sum_{j\geq 0}\frac{(v/y)^jq^{j^2}}{(q^2;q^2)_j}\sum_{k\geq 0}\frac{(wx/y)^k}{(q^2;q^2)_k}\sum_{l\geq 0}\frac{(t/xy)^l}{(q^2;q^2)_l} \nonumber\\
&\quad \times \sum_{m=-\infty}^{\infty}x^{-m}q^{m^2}\sum_{n=-\infty}^{\infty}y^{-n}q^{n^2} \nonumber\\
&=\CT_{y}(-quy,-qv/y;q^2)_{\infty}(-qy,-q/y,q^2;q^2)_{\infty} \CT_{x}\frac{(-qx,-q/x,q^2;q^2)_{\infty}}{(wx/y,t/(xy);q^2)_{\infty}} \nonumber\\
&=\CT_{y} (-quy,-qv/y;q^2)_{\infty}(-qy,-q/y,q^2;q^2)_{\infty} (q^2;q^2)_\infty \nonumber \\
&\qquad \times \CT_x \sum_{i=0}^\infty \frac{(-qy/w;q^2)_i(wx/y)^i}{(q^2;q^2)_i} \sum_{j=0}^\infty \frac{(-qy/t;q^2)_j(t/(xy))^j}{(q^2;q^2)_j} \quad \text{(by \eqref{q-binomial})} \nonumber \\
&=\CT_y (-quy,-qv/y;q^2)_{\infty}(-qy,-q/y,q^2,q^2;q^2)_{\infty} \sum_{i=0}^\infty \frac{(-qy/w,-qy/t;q^2)_i(tw/y^2)^i}{(q^2,q^2;q^2)_i} \nonumber \\
&=\CT_{y}\frac{(-quy,-qv/y,-qw/y,-qt/y;q^2)_{\infty}(-qy,-q/y,q^2;q^2)_{\infty}}{(wt/y^2;q^2)_{\infty}}. \quad \text{(by \eqref{Gauss})} \nonumber
\end{align}

(1) Let $v=1/(qu)$, $w=1/u$ and $t=1/(q^2u)$ in \eqref{CW-thm6.2-ij3-in-1}. We have
\begin{align}
&F(u,1/(qu),1/u,1/(q^2u);q^2) \nonumber \\
&=\CT_{y}\frac{(-quy,-1/(uy),-q/(uy),-1/(quy);q^2)_{\infty}(-qy,-q/y,q^2;q^2)_{\infty}}{(1/(q^2u^2y^2);q^2)_{\infty}} \nonumber\\
&=\CT_{y}\frac{(-quy,-q/(uy);q^2)_{\infty}(-qy,-q/y,q^2;q^2)_{\infty}}{(1/(quy);q)_{\infty}} \nonumber\\
&=\frac{1}{(q^2;q^2)_\infty}\CT_{y}\sum_{i\geq 0}\frac{(quy)^{-i}}{(q;q)_{i}}\sum_{m=-\infty}^{\infty}(uy)^{m}q^{m^2} \sum_{n=-\infty}^\infty y^{-n}q^{n^2}  \nonumber\\
&=\frac{1}{(q^2;q^2)_\infty}\sum_{i\geq 0}\frac{(qu)^{-i}}{(q;q)_{i}}\sum_{m=-\infty}^{\infty}u^{m}q^{m^2+(m-i)^2} \nonumber\\
&=\frac{1}{(q^2;q^2)_\infty}\Big(\sum_{i\geq 0}\frac{u^{-i}q^{2i^2-2i}}{(q;q)_{2i}}\sum_{m=-\infty}^{\infty}u^{m-i}q^{2(m-i)^2}  \nonumber \\
&\qquad +
\sum_{i\geq 0}\frac{u^{-i-1}q^{2i^2}}{(q;q)_{2i+1}}\sum_{m=-\infty}^{\infty}u^{m-i}q^{2(m-i)^2-2(m-i)}\Big) \nonumber\\
&=\frac{1}{(q^2;q^2)_\infty}\Big((-q^2u,-q^2/u,q^4;q^4)_{\infty}\sum_{i\geq 0}\frac{u^{-i}q^{2i^2-2i}}{(q;q)_{2i}} \nonumber \\
&\qquad +
(-u,-q^4/u,q^4;q^4)_{\infty}\sum_{i\geq 0}\frac{u^{-i-1}q^{2i^2}}{(q;q)_{2i+1}}\Big). \label{thm3-case1-start}
\end{align}

Setting $u=q^{-2}$ in \eqref{thm3-case1-start} and then using  \eqref{S. 38} and \eqref{S. 39}, we obtain \eqref{thm3-1}.

Setting $u=q^{-1}$ in \eqref{thm3-case1-start} and then using \eqref{S. 9} and \eqref{S. 85}, we obtain \eqref{thm3-2}.

Setting $u=q^{-3}$, we have
\begin{align}
&F(q^{-3},q^{2},q^{3},q;q^2) \nonumber \\
&=\frac{1}{(q^2;q^2)_\infty}\Big((-q^5,-q^{-1},q^4;q^4)_{\infty}\sum_{i\geq 0}\frac{q^{2i^2+i}}{(q;q)_{2i}}+(-q^7,-q^{-3},q^4;q^4)_{\infty}\sum_{i\geq 0}\frac{q^{2i^2+3i+3}}{(q;q)_{2i+1}}\Big) \nonumber \\
&=q^{-1}\frac{(-q,-q^3,q^4;q^4)_{\infty}}{(q^2;q^2)_\infty}\sum_{i\geq 0}\frac{q^{(i^2+i)/2}}{(q;q)_{i}}=q^{-1}\frac{J_{2}^{2}}{J_{1}^{2}}.
\end{align}
Here for the last equality we used \eqref{Euler2}. This proves \eqref{thm3-3}.

(2) Let $v=1/qu$, $w=1/u$ and $t=1/u$ in \eqref{CW-thm6.2-ij3-in-1}. We have
\begin{align}
&F(u,1/(qu),1/u,1/u;q^2) \nonumber \\
&=\CT_{y}\frac{(-quy,-1/(uy),-q/(uy),-q/(uy);q^2)_{\infty}(-qy,-q/y,q^2;q^2)_{\infty}}{(1/(u^2y^2);q^2)_{\infty}} \nonumber\\
&=\CT_{y}\frac{(-quy,-q/(uy);q^2)_{\infty}(-qy,-q/y,q^2;q^2)_{\infty}}{(1/(uy);q)_{\infty}} \nonumber\\
&=\frac{1}{(q^2;q^2)_\infty}\CT_{y}\sum_{i\geq 0}\frac{(uy)^{-i}}{(q;q)_{i}}\sum_{m=-\infty}^{\infty}(uy)^{m}q^{m^2} \sum_{n=-\infty}^{\infty} y^{-n}q^{n^2} \nonumber\\
&=\frac{1}{(q^2;q^2)_\infty}\sum_{i\geq 0}\frac{u^{-i}}{(q;q)_{i}}\sum_{m=-\infty}^{\infty}u^{m}q^{m^2+(m-i)^2} \nonumber\\
&=\frac{1}{(q^2;q^2)_\infty}\Big(\sum_{i\geq 0}\frac{u^{-i}q^{2i^2}}{(q;q)_{2i}}\sum_{m=-\infty}^{\infty}u^{m-i}q^{2(m-i)^2} \nonumber \\
&\qquad +
\sum_{i\geq 0}\frac{u^{-i-1}q^{2i^2+2i+1}}{(q;q)_{2i+1}}\sum_{m=-\infty}^{\infty}u^{m-i}q^{2(m-i)^2-2(m-i)}\Big) \nonumber\\
&=\frac{1}{(q^2;q^2)_\infty}\Big((-q^2u,-q^2/u,q^4;q^4)_{\infty}\sum_{i\geq 0}\frac{u^{-i}q^{2i^2}}{(q;q)_{2i}} \nonumber \\
&\qquad +
(-u,-q^4/u,q^4;q^4)_{\infty}\sum_{i\geq 0}\frac{u^{-i-1}q^{2i^2+2i+1}}{(q;q)_{2i+1}}\Big). \label{thm3-2-start}
\end{align}

Setting $u=q$ and then using \eqref{S. 9} and \eqref{S. 85}, we obtain \eqref{thm3-4}.

Setting $u=1$ and then using \eqref{S. 38} and \eqref{S. 39}, we obtain \eqref{thm3-5}.

Setting $u=q^{-1}$ in \eqref{thm3-2-start}, we have
\begin{align}
&F(q^{-1},1,q,q;q^2) \nonumber\\
&=\frac{1}{(q^2;q^2)_\infty}\Big((-q,-q^{3},q^4;q^4)_{\infty}\sum_{i\geq 0}\frac{q^{2i^2+i}}{(q;q)_{2i}}+(-q^5,-q^{-1},q^4;q^4)_{\infty}\sum_{i\geq 0}\frac{q^{2i^2+3i+2}}{(q;q)_{2i+1}}\Big) \nonumber\\
&=\frac{(-q,-q^3,q^4;q^4)_{\infty}}{(q^2;q^2)_\infty}\sum_{i\geq 0}\frac{q^{(i^2+i)/2}}{(q;q)_{i}}=\frac{J_{2}^{2}}{J_{1}^{2}}.
\end{align}
Here for the last equality we used \eqref{Euler2}. This proves \eqref{thm3-6}.

(3) By \eqref{CW-thm6.2-ij3-in-1}, we have
\begin{align}
&F(q^{-3},q^3,q^3,q^{2};q^2)=\CT_{y}\frac{(-y/q^2,-q^4/y,-q^4/y,-q^3/y;q^2)_{\infty}(-qy,-q/y,q^2;q^2)_{\infty}}{(q^5/y^2;q^2)_{\infty}} \nonumber\\
&=\CT_{y}\frac{(-q^3/y;q)_\infty(-y/q^2,-q^4/y;q^2)_{\infty}(-qy,-q/y,q^2;q^2)_{\infty}}{(q^5/y^2;q^2)_{\infty}} \nonumber\\
&=\frac{1}{(q^2;q^2)_\infty}\CT_{y}\sum_{i,j\geq 0}\frac{y^{-i-2j}q^{(i^2+5i)/2+5j}}{(q;q)_{i}(q^2;q^2)_{j}}\sum_{m=-\infty}^{\infty}y^{-m}q^{m^2+3m} \sum_{n=-\infty}^{\infty} y^{n}q^{n^2}  \nonumber\\
&=\frac{1}{(q^2;q^2)_\infty}\sum_{i,j\geq 0}\frac{q^{(i^2+5i)/2+5j}}{(q;q)_{i}(q^2;q^2)_{j}}\sum_{m=-\infty}^{\infty}q^{m^2+3m+(i+2j+m)^2}  \nonumber\\
&=\frac{1}{(q^2;q^2)_\infty}\Big(\sum_{i,j\geq 0}\frac{q^{4i^2+2j^2+4ij+2i+2j}}{(q;q)_{2i}(q^2;q^2)_{j}}\sum_{m=-\infty}^{\infty}q^{2(m+i+j)^2+3(m+i+j)}  \nonumber\\
&\quad +\sum_{i,j\geq 0}\frac{q^{4i^2+2j^2+4ij+6i+4j+4}}{(q;q)_{2i+1}(q^2;q^2)_{j}}\sum_{m=-\infty}^{\infty}q^{2(m+i+j)^2+5(m+i+j)}\Big)  \nonumber\\
&=\frac{1}{(q^2;q^2)_\infty}\Big((-q^5,-q^{-1},q^4;q^4)_{\infty}\sum_{i,j\geq 0}\frac{q^{4i^2+2j^2+4ij+2i+2j}}{(q;q)_{2i}(q^2;q^2)_{j}}  \nonumber\\
&\quad +(-q^7,-q^{-3},q^4;q^4)_{\infty}\sum_{i,j\geq 0}\frac{q^{4i^2+2j^2+4ij+6i+4j+4}}{(q;q)_{2i+1}(q^2;q^2)_{j}}\Big)  \nonumber\\
&=\frac{(-q,-q^3,q^4;q^4)_{\infty}}{(q^2;q^2)_\infty}\sum_{i,j\geq 0}\frac{q^{i^2+2j^2+2ij+i+2j-1}}{(q;q)_{i}(q^2;q^2)_{j}}  \nonumber\\
&=\frac{q^{-1}J_{6}^2}{J_{1}J_{3}}.
\end{align}
Here for the last equality we used the following particular case of Bressoud's identity \cite{Bressoud1980}:
\begin{align}\label{eq-Bressoud-case}
\sum_{i,j\geq 0}\frac{q^{i^2+2j^2+2ij+i+2j}}{(q;q)_{i}(q^2;q^2)_{j}}=\frac{J_{6}^2}{J_{2}J_{3}}.
\end{align}
This proves \eqref{thm3-7}.

(4) By \eqref{CW-thm6.2-ij3-in-1} we have
\begin{align}
&F(q^{-1},q,q,q^2;q^2)=\CT_{y}\frac{(-q^2/y;q)_{\infty}(-y,-q^2/y,-qy,-q/y,q^2;q^2)_{\infty}}{(q^3/y^2;q^2)_{\infty}} \nonumber\\
&=\frac{1}{(q^2;q^2)_\infty}\CT_{y}\sum_{i,j\geq 0}\frac{y^{-i-2j}q^{(i^2+3i)/2+3j}}{(q;q)_{i}(q^2;q^2)_{j}}\sum_{m=-\infty}^{\infty}y^{-m}q^{m^2+m} \sum_{n=-\infty}^{\infty} y^{n}q^{n^2}  \nonumber\\
&=\frac{1}{(q^2;q^2)_\infty}\sum_{i,j\geq 0}\frac{q^{(i^2+3i)/2+3j}}{(q;q)_{i}(q^2;q^2)_{j}}\sum_{m=-\infty}^{\infty}q^{m^2+m+(i+2j+m)^2} \nonumber\\
&=\frac{1}{(q^2;q^2)_\infty}\Big(\sum_{i,j\geq 0}\frac{q^{4i^2+2j^2+4ij+2i+2j}}{(q;q)_{2i}(q^2;q^2)_{j}}\sum_{m=-\infty}^{\infty}q^{2(m+i+j)^2+(m+i+j)}  \nonumber\\
&\quad  +\sum_{i,j\geq 0}\frac{q^{4i^2+2j^2+4ij+6i+4j+3}}{(q;q)_{2i+1}(q^2;q^2)_{j}}\sum_{m=-\infty}^{\infty}q^{2(m+i+j)^2+3(m+i+j)}\Big) \nonumber\\
&=\frac{1}{(q^2;q^2)_\infty}\Big((-q,-q^{3},q^4;q^4)_{\infty}\sum_{i,j\geq 0}\frac{q^{4i^2+2j^2+4ij+2i+2j}}{(q;q)_{2i}(q^2;q^2)_{j}} \nonumber\\
&\quad  +(-q^5,-q^{-1},q^4;q^4)_{\infty}\sum_{i,j\geq 0}\frac{q^{4i^2+2j^2+4ij+6i+4j+3}}{(q;q)_{2i+1}(q^2;q^2)_{j}}\Big) \nonumber\\
&=\frac{(-q,-q^3,q^4;q^4)_{\infty}}{(q^2;q^2)_\infty}\sum_{i,j\geq 0}\frac{q^{i^2+2j^2+2ij+i+2j}}{(q;q)_{i}(q^2;q^2)_{j}}.
\end{align}
Now by \eqref{eq-Bressoud-case} we obtain \eqref{thm3-8}.

(5) Let $w=1/u$ in \eqref{CW-thm6.2-ij3-in-1}. We have
\begin{align}\label{CW-thm6.2-ij3-in-2}
&F(u,v,1/u,t;q^2)=\CT_{y}\frac{(-quy,-qv/y,-q/(uy),-qt/y;q^2)_{\infty}(-qy,-q/y,q^2;q^2)_{\infty}}{(t/(uy^2);q^2)_{\infty}} \nonumber\\
&=\frac{1}{(q^2;q^2)_\infty}\CT_{y}\sum_{i\geq 0}\frac{(t/uy^2)^{i}}{(q^2;q^2)_{i}}\sum_{j\geq 0}\frac{(v/y)^{j}q^{j^2}}{(q^2;q^2)_{j}}\sum_{k\geq 0}\frac{(t/y)^{k}q^{k^2}}{(q^2;q^2)_{k}} \nonumber \\
&\qquad \qquad \times \sum_{m=-\infty}^{\infty}(uy)^{-m}q^{m^2}\sum_{s=-\infty}^{\infty}y^{s}q^{s^2}\nonumber\\
&=\frac{1}{(q^2;q^2)_\infty}\sum_{i,j,k\geq 0}\frac{(t/u)^{i}v^{j}t^{k}q^{j^2+k^2}}{(q^2;q^2)_{i}(q^2;q^2)_{j}(q^2;q^2)_{k}}\sum_{m=-\infty}^{\infty}u^{-m}q^{m^2+(2i+j+k+m)^{2}} \nonumber \\
&=\frac{1}{(q^2;q^2)_\infty}\sum_{i,n\geq 0}\sum_{j=0}^{n}\frac{(t/u)^{i}v^{j}t^{n-j}q^{j^2+(n-j)^2}}{(q^2;q^2)_{i}(q^2;q^2)_{j}(q^2;q^2)_{n-j}}\sum_{m=-\infty}^{\infty}u^{-m}q^{m^2+(2i+n+m)^{2}}\nonumber\\
&=\frac{1}{(q^2;q^2)_\infty}(S_{0}(q)+S_{1}(q)),
\end{align}
where $S_{0}(q)$ and $S_{1}(q)$ correspond to the sum with $n$ even and odd, respectively.

For $n$ even we write $n=2n_1$. We have
\begin{align}\label{CW-thm6.2-ij3-in-2-1}
&S_{0}(q)=\sum_{i,n_{1}\geq 0}\sum_{j=0}^{2n_{1}}\frac{u^{n_{1}}v^{j}t^{i+2n_{1}-j}q^{j^2+(2n_{1}-j)^2}}{(q^2;q^2)_{i}(q^2;q^2)_{j}(q^2;q^2)_{2n_{1}-j}}\sum_{m=-\infty}^{\infty}u^{-(m+i+n_{1})}q^{2(m+i+n_{1})^2+2(i+n_{1})^{2}}\nonumber\\
&=(-q^2u,-q^2/u,q^4;q^4)_{\infty}\sum_{n_{2}=0}^{\infty}\sum_{n_{1}=0}^{n_{2}}\sum_{j=0}^{2n_{1}}\frac{t^{n_{2}}(ut)^{n_{1}}(v/t)^{j}q^{j^2+(2n_{1}-j)^2+2n_{2}^{2}}}{(q^2;q^2)_{n_{2}-n_{1}}(q^2;q^2)_{j}(q^2;q^2)_{2n_{1}-j}}  \nonumber \\
&\qquad \qquad \qquad \qquad \qquad \qquad \qquad \qquad  \qquad  \qquad \qquad \qquad \qquad \qquad\text{(\text{set $i+n_{1}=n_{2}$})}\nonumber\\
&=(-q^2u,-q^2/u,q^4;q^4)_{\infty}T_0(u,v,t;q^2),
\end{align}
where
\begin{align}\label{add-T0}
&T_0(u,v,t;q^2)=\sum_{n_{2}=0}^{\infty}\sum_{n_{1}=0}^{n_{2}}\frac{(uv)^{n_{1}}t^{n_{2}}q^{2n_{1}^{2}+2n_{2}^{2}}}{(q^2;q^2)_{n_{2}-n_{1}}} \times \Big(\sum_{r=0}^{n_{1}}\frac{q^{2r^2}(t/v)^{r}}{(q^2;q^2)_{n_{1}-r}(q^2;q^2)_{n_{1}+r}} \nonumber \\
&\qquad +\sum_{r=1}^{n_{1}}\frac{q^{2r^2}(v/t)^{r}}{(q^2;q^2)_{n_{1}-r}(q^2;q^2)_{n_{1}+r}}\Big).
\end{align}
Here for the last equality in \eqref{CW-thm6.2-ij3-in-2-1} we write $j=n_1-r$ ($0\leq r\leq n_1$) and $j=n_1+r$ ($1\leq r \leq n_1$).

Similarly, for $n$ odd we write $n=2n_1+1$. We have
\begin{align}\label{CW-thm6.2-ij3-in-2-2}
&S_{1}(q)=\sum_{i,n_{1}\geq 0}\sum_{j=0}^{2n_{1}+1}\frac{u^{n_{1}}v^{j}t^{i+2n_{1}-j+1}q^{j^2+(2n_{1}-j+1)^2}}{(q^2;q^2)_{i}(q^2;q^2)_{j}(q^2;q^2)_{2n_{1}-j+1}}\nonumber\\
&\quad\times \sum_{m=-\infty}^{\infty}u^{-(m+i+n_{1})}q^{2(m+i+n_{1})^2+2(m+i+n_{1})+2(i+n_{1})^{2}+2(i+n_{1})+1} \quad \text{(set $i+n_{1}=n_{2}$)}\nonumber\\
&=t(-u,-q^4/u,q^4;q^4)_{\infty}\sum_{n_{2}=0}^{\infty}\sum_{n_{1}=0}^{n_{2}}\sum_{j=0}^{2n_{1}+1}\frac{t^{n_{2}}(ut)^{n_{1}}(v/t)^{j}q^{j^2+(2n_{1}-j+1)^2+2n_{2}^{2}+2n_{2}+1}}{(q^2;q^2)_{n_{2}-n_{1}}(q^2;q^2)_{j}(q^2;q^2)_{2n_{1}-j+1}} \nonumber\\
&=tq^2(-u,-q^4/u,q^4;q^4)_{\infty}T_1(u,v,t;q^2),
\end{align}
where
\begin{align}
T_1(u,v,t;q^2)&=\sum_{n_{2}=0}^{\infty}\sum_{n_{1}=0}^{n_{2}}\frac{(uv)^{n_{1}}t^{n_{2}}q^{2n_{1}^{2}+2n_{1}+2n_{2}^{2}+2n_{2}}}{(q^2;q^2)_{n_{2}-n_{1}}} \sum_{r=0}^{n_{1}}\frac{q^{2r^2+2r}((t/v)^r+(v/t)^{r+1})}{(q^2;q^2)_{n_{1}-r}(q^2;q^2)_{n_{1}+r+1}}.
\end{align}
Here for the last equality in \eqref{CW-thm6.2-ij3-in-2-2} we write $j=n_1-r$ ($0\leq r\leq n_1$) and $j=n_1+r+1$ ($0\leq r \leq n_1$).

Substituting \eqref{CW-thm6.2-ij3-in-2-1} and \eqref{CW-thm6.2-ij3-in-2-2} into \eqref{CW-thm6.2-ij3-in-2}, we have
\begin{align}\label{CW-thm6.2-ij3-in-3}
&F(u,v,1/u,t;q^2)=\frac{1}{(q^2;q^2)_\infty}\Big((-q^2u,-q^2/u,q^4;q^4)_{\infty}T_0(u,v,t;q^2) \nonumber \\
&\qquad +tq^2(-u,-q^4/u,q^4;q^4)_{\infty}T_1(u,v,t;q^2)\Big).
\end{align}

The desired identity \eqref{thm3-9} corresponds to the case $v=1/u$ and $t=1$.

We define a Bailey pair $(\alpha_r^{(0)}(1;q^2),\beta_r^{(0)}(1;q^2))$ with
\begin{align}
\alpha_{0}^{(0)}(1;q^2)=1,\quad \alpha_{r}^{(0)}(1;q^2)=q^{2r^2}(u^r+u^{-r}).
\end{align}
Applying \eqref{Bailey's lemma-1.1} to $(\alpha_r^{(0)},\beta_r^{(0)})$ we obtain a new Bailey pair:
\begin{align}
&\alpha_0^{(1)}(1;q^2)=1, \quad \alpha_r^{(1)}(1;q^2)=q^{4r^2}(u^r+u^{-r}), \\
&\beta_r^{(1)}(1;q^2)=\sum_{k=0}^r \frac{q^{2k^2}}{(q^2;q^2)_{r-k}}\beta_k^{(0)}(1;q^2).
\end{align}
Substituting this Bailey pair into \eqref{Bailey's lemma-1.2}, we deduce that
\begin{align}\label{BP-S1-2th-1}
&T_0(u,1/u,1;q^2)=\sum_{n_{2}=0}^{\infty}\sum_{n_{1}=0}^{n_{2}}\frac{q^{2n_{1}^{2}+2n_{2}^{2}}}{(q^2;q^2)_{n_{2}-n_{1}}}\beta_{n_{1}}^{(0)}(1;q^2)=\frac{1}{(q^2;q^2)_{\infty}}\sum_{n=-\infty}^{\infty}q^{6n^2}u^n \nonumber \\
&=\frac{(-q^6u,-q^6/u,q^{12};q^{12})_\infty}{(q^2;q^2)_\infty}.
\end{align}
Similarly, we define a Bailey pair $(\widetilde{\alpha}_r^{(0)}(q^2;q^2),\widetilde{\beta}_r^{(0)}(q^2;q^2))$ where
\begin{align}
\widetilde{\alpha}_{r}^{(0)}(q^2;q^2)=q^{2r^2+2r}(u^r+u^{-r-1}).
\end{align}
Applying \eqref{Bailey's lemma-1.1} to it to generate a new Bailey pair, and then by \eqref{Bailey's lemma-1.2}, we have
\begin{align}\label{BP-S1-2th-q1}
&T_1(u,1/u,1;q^2)=\frac{1}{1-q^2}\sum_{n_{2}=0}^{\infty}\sum_{n_{1}=0}^{n_{2}}\frac{q^{2n_{1}^{2}+2n_{1}+2n_{2}^{2}+2n_{2}}}{(q^2;q^2)_{n_{2}-n_{1}}}\widetilde{\beta}_{n_{1}}^{(0)}(q^2;q^2) \nonumber \\
&=\frac{1}{(q^2;q^2)_{\infty}}\sum_{n=-\infty}^{\infty}q^{6n^2+6n}u^n=\frac{(-q^{12}u,-1/u,q^{12};q^{12})_\infty}{(q^2;q^2)_\infty}.
\end{align}

Substituting \eqref{BP-S1-2th-1} and \eqref{BP-S1-2th-q1} into \eqref{CW-thm6.2-ij3-in-3}, we obtain \eqref{thm3-9}.

(7) Let $(u,v,t)=(1,q^2,q^2)$ in  \eqref{CW-thm6.2-ij3-in-3}.

Let $(\alpha_n^{(0)}(1;q),\beta_n^{(0)}(1;q))$ be a Bailey pair with
\begin{align}
\alpha_{0}^{(0)}(1;q)=1,\quad \alpha_{n}^{(0)}(1;q)=2q^{n^2}.
\end{align}
Applying \eqref{Lovejoy-a-aq-1} to it, we obtain the Bailey pair
\begin{equation}
\begin{split}
&\alpha_{n}^{(1)}(q;q)=(2n+1)\frac{q^{n^2}(1-q^{2n+1})}{1-q}, \\
&\beta_{n}^{(1)}(q;q)=\beta_{n}^{(0)}(1;q).
\end{split}
\end{equation}
Applying \eqref{Bailey's lemma-1.1} to $(\alpha_n^{(1)},\beta_n^{(1)})$ to generate a new Bailey pair $(\alpha_n^{(2)}(q;q),\beta_n^{(2)}(q;q))$, and then substituting $(\alpha_n^{(2)},\beta_n^{(2)})$ into \eqref{Bailey's lemma-1.2}, we deduce that
\begin{align}\label{BP-LS1-2th-1}
&T_0(1,q,q;q)=\sum_{n_{2}=0}^{\infty}\sum_{n_{1}=0}^{n_{2}}\frac{q^{n_{1}^{2}+n_{1}+n_{2}^{2}+n_{2}}}{(q;q)_{n_{2}-n_{1}}}\beta_{n_{1}}^{(1)}(q;q) \nonumber \\
&=\frac{1}{(q;q)_{\infty}}\sum_{n=-\infty}^{\infty}(2n+1)q^{3n^2+2n}.
\end{align}

 Let $(\widetilde{\alpha}_n^{(0)}(q;q),\widetilde{\beta}_n^{(0)}(q;q))$ be a Bailey pair with
\begin{align}
\widetilde{\alpha}_{n}^{(0)}(q;q)=q^{n^2+n}.
\end{align}
Applying \eqref{Lovejoy-a-aq-1} to it, we obtain the Bailey pair
\begin{equation}
\begin{split}
&\widetilde{\alpha}_{n}^{(1)}(q^2;q)=(n+1)\frac{q^{n^2+n}(1-q^{2n+2})}{1-q^2}, \\
&\widetilde{\beta}_{n}^{(1)}(q^2;q)=\beta_{n}^{(0)}(q;q).
\end{split}
\end{equation}
Applying \eqref{Bailey's lemma-1.1} to $(\widetilde{\alpha}_n^{(1)},\widetilde{\beta}_n^{(1)})$ to generate a new Bailey pair $(\widetilde{\alpha}_n^{(2)}(q^2;q),\widetilde{\beta}_n^{(2)}(q^2;q))$, and then substituting $(\widetilde{\alpha}_n^{(2)},\widetilde{\beta}_n^{(2)})$ into \eqref{Bailey's lemma-1.2}, we deduce that
\begin{align}\label{BP-LS1-2th-2}
&T_1(1,q,q;q)=\frac{2}{1-q}\sum_{n_{2}=0}^{\infty}\sum_{n_{1}=0}^{n_{2}}\frac{q^{n_{1}^{2}+2n_{1}+n_{2}^{2}+2n_{2}}}{(q;q)_{n_{2}-n_{1}}}\widetilde{\beta}_{n_{1}}^{(1)}(q^2;q) \nonumber \\
&=\frac{2}{(q;q)_{\infty}}\sum_{n=0}^{\infty}(n+1)(1-q^{2n+2})q^{3n^2+5n} \nonumber\\
&=\frac{2}{(q;q)_{\infty}}\sum_{n=-\infty}^{\infty}(n+1)q^{3n^2+5n}=-\frac{2}{(q;q)_{\infty}}\sum_{n=-\infty}^{\infty}nq^{3n^2+n-2}.
\end{align}

We recall the following two identities (see e.g. \cite[Corollaries 1.3.21 and 1.3.22]{Berndt2006}):
\begin{align}
\sum_{n=-\infty}^{\infty}(6n+1)q^{3n^2+n}&=\frac{J_{2}^{5}}{J_{4}^{2}},   \label{Berndt-1} \\
\sum_{n=-\infty}^{\infty}(3n+1)q^{3n^2+2n}&=\frac{J_{1}^{2}J_{4}^{2}}{J_{2}}.  \label{Berndt-2}
\end{align}

Substituting \eqref{BP-LS1-2th-1} and \eqref{BP-LS1-2th-2} with $q$ replaced by $q^2$ into \eqref{CW-thm6.2-ij3-in-3}, we deduce that
\begin{align}\label{CW-thm6.2-ij3-in-3-4}
&F(1,q^2,1,q^2;q^2)=\frac{1}{(q^2;q^2)_\infty}\Big( (-q^2,-q^2,q^4;q^4)_\infty T_0(1,q^2,q^2;q^2) \nonumber \\
&\qquad +q^4(-1,-q^4,q^4;q^4)_\infty T_1(1,q^2,q^2;q^2)\Big) \nonumber \\
&=\frac{1}{J_{2}^{2}}\Big(\frac{J_{4}^{5}}{J_{2}^{2}J_{8}^{2}}\sum_{n=-\infty}^{\infty}(2n+1)q^{6n^2+4n}-\frac{4J_{8}^{2}}{J_{4}}\sum_{n=-\infty}^{\infty}nq^{6n^2+2n}\Big) \nonumber\\
&=\frac{1}{J_{2}^{2}}\Big(\frac{J_{4}^{5}}{J_{2}^{2}J_{8}^{2}}\Big(\frac{2}{3}\sum_{n=-\infty}^{\infty}(3n+1)q^{6n^2+4n}+\frac{1}{3}\sum_{n=-\infty}^{\infty}q^{6n^2+4n}\Big) \nonumber\\
&\quad +\frac{2J_{8}^{2}}{3J_{4}}\Big(\sum_{n=-\infty}^{\infty}q^{6n^2+2n}-\sum_{n=-\infty}^{\infty}(6n+1)q^{6n^2+2n} \Big)\Big)\nonumber\\
&=\frac{1}{J_{2}^{2}}\Big(\frac{J_{4}^{5}}{J_{2}^{2}J_{8}^{2}}\Big(\frac{2J_{2}^{2}J_{8}^{2}}{3J_{4}}+\frac{J_{12}^{2}J_{4,24}}{3J_{2,12}J_{24}}\Big) +\frac{2J_{8}^{2}}{3J_{4}}\Big(\frac{J_{8}J_{12}^{2}}{J_{4}J_{24}}-\frac{J_{4}^{5}}{J_{8}^{2}} \Big)\Big) \nonumber\\
&=\frac{J_{12}^{2}}{3J_{2}^{2}J_{24}}\Big(\frac{J_{4}^{5}J_{4,24}}{J_{2}^{2}J_{8}^{2}J_{2,12}}+\frac{2J_{8}^{3}}{J_{4}^{2}} \Big).
\end{align}
Here for the last second equality we used \eqref{Jacobi}, \eqref{Berndt-1} and \eqref{Berndt-2}.
This proves \eqref{thm3-10}.

(8) By definition we have
\begin{align}\label{CW-thm6.2-ij3-in-3-5}
&F(q^{-2},q^2,1,q^2;q^2)-F(1,q^2,1,q^2;q^2) \nonumber\\
&=\sum_{i,j,k,l\geq 0} \frac{q^{2i^2+2j^2+2k^2+2l^2-2ij-2ik-2il+2jk+2jl-2i+2j+2l}(1-q^{2i})}{(q^2;q^2)_i(q^2;q^2)_j(q^2;q^2)_k(q^2;q^2)_l}\nonumber\\
&=\sum_{i,j,k,l\geq 0} \frac{q^{2(i+1)^2+2j^2+2k^2+2l^2-2(i+1)j-2(i+1)k-2(i+1)l+2jk+2jl-2(i+1)+2j+2l}}{(q^2;q^2)_i(q^2;q^2)_j(q^2;q^2)_k(q^2;q^2)_l}\nonumber\\
&=\sum_{i,j,k,l\geq 0} \frac{q^{2i^2+2j^2+2k^2+2l^2-2ij-2ik-2il+2jk+2jl+2i-2k}}{(q^2;q^2)_i(q^2;q^2)_j(q^2;q^2)_k(q^2;q^2)_l}\nonumber\\
&=F(q^2,1,q^{-2},1;q^2).
\end{align}
By \eqref{thm3-symmetry} and \eqref{thm3-9} we have
\begin{align}\label{add-thm3-mid}
F(q^{-2},q^2,1,q^2;q^2)=F(q^{-2},q^2,q^2,1;q^2)=\frac{J_{12}^{2}}{J_{2}^{2}J_{24}}\left(\frac{J_{4}^{5}J_{4,24}}{J_{2}^{2}J_{8}^{2}J_{2,12}}+\frac{2J_{8}^{3}}{J_{4}^{2}} \right).
\end{align}
Substituting \eqref{add-thm3-mid} and \eqref{CW-thm6.2-ij3-in-3-4} into \eqref{CW-thm6.2-ij3-in-3-5}, we obtain \eqref{thm3-11}.

(9) Setting $(u,v,t)=(q^2,q^{-2},q^{-2})$ in \eqref{CW-thm6.2-ij3-in-3}, we have
\begin{align}\label{CW-thm6.2-ij3-in-3-6}
&F(q^2,q^{-2},q^{-2},q^{-2};q^2)=\frac{1}{(q^2;q^2)_\infty}\Big((-1,-q^4,q^4;q^4)_{\infty}T_0(q^2,q^{-2},q^{-2};q^2) \nonumber \\
&\qquad +2(-q^2,-q^2,q^4;q^4)_{\infty}T_1(q^2,q^{-2},q^{-2};q^2)\Big).
\end{align}

Let $(\alpha_n^{(0)}(1;q),\beta_n^{(0)}(1;q))$ be a Bailey pair with
\begin{align}
\alpha_{0}^{(0)}(1;q)=1,\quad \alpha_{n}^{(0)}(1;q)=2q^{n^2}.
\end{align}
Applying \eqref{Bailey's lemma-1.1} to it, we obtain the Bailey pair
\begin{equation}\label{W-BP-a-1}
\begin{split}
&\alpha_{0}^{(1)}(1;q)=1,\quad  \alpha_{n}^{(1)}(1;q)=2q^{2n^2}, \\
&\beta_{n}^{(1)}(1;q)=\sum_{k=0}^{n}\frac{q^{k^2}}{(q;q)_{n-k}}\beta_{k}^{(0)}(1;q).
\end{split}
\end{equation}
Applying \eqref{W-BP-a-aq} to \eqref{W-BP-a-1}, we obtain the Bailey pair:
\begin{equation}\label{W-BP-a-2}
    \begin{split}
        \alpha_n^{(2)}(q;q)&=\frac{(1-q^{2n+1})q^{-n}}{1-q} \Big(1+2\sum_{r=1}^n q^{2r^2}\Big), \\
        \beta_n^{(2)}(q;q)&=q^{-n}\beta_n^{(1)}(1;q).
    \end{split}
\end{equation}
Applying \eqref{McLaughlin-a-a/q} to \eqref{W-BP-a-2}, we obtain the Bailey pair: $\alpha_0^{(3)}(1;q)=1$ and for $n\geq 1$,
\begin{equation}\label{W-BP-a-3}
    \begin{split}
    \alpha_n^{(3)}(1;q)&=q^{-n}\Big(1+2\sum_{r=1}^nq^{2r^2}\Big)-q^n\Big(1+2\sum_{r=1}^{n-1}q^{2r^2}\Big), \\
        \beta_n^{(3)}(1;q)&=\beta_n^{(2)}(q;q).
    \end{split}
\end{equation}
We can now evaluate the first sum in \eqref{CW-thm6.2-ij3-in-3-6}. Applying \eqref{Bailey's lemma-1.2} with the Bailey pair $(\alpha_n^{(3)}(1,q),\beta_n^{(3)}(1,q))$, we deduce that
\begin{align}
&T_0(q,q^{-1},q^{-1};q)
=\sum_{n=0}^\infty q^{n^2}\beta_n^{(3)}(1;q)=\frac{1}{(q;q)_\infty}\sum_{n=0}^\infty q^{n^2}\alpha_n^{(3)}(1;q) \nonumber \\
&=\frac{1}{(q;q)_\infty}\Big(1+ \sum_{n=1}^\infty q^{n^2-n}\big(1+2\sum_{r=1}^n q^{2r^2}\big)-\sum_{n=1}^\infty q^{n^2+n} \big(1+2\sum_{r=1}^{n-1} q^{2r^2}\big)\Big) \nonumber \\
&=\frac{2}{(q;q)_\infty}\Big(1+ \sum_{n=1}^\infty q^{n^2-n}\sum_{r=1}^n q^{2r^2}-\sum_{n=2}^\infty q^{n^2-n}\sum_{r=1}^{n-2} q^{2r^2}\Big) \nonumber \\
&=\frac{2}{(q;q)_\infty}\Big(1+q^2+ \sum_{n=2}^\infty q^{n^2-n}(q^{2n^2}+q^{2(n-1)^2})\Big)\nonumber \\
&=\frac{2}{(q;q)_\infty}\sum_{n=1}^\infty q^{n^2-n}(q^{2n^2}+q^{2(n-1)^2})
\nonumber \\
&=\frac{2}{(q;q)_\infty} \sum_{n=-\infty}^\infty q^{3n^2+n}=2\frac{(-q^2,-q^4,q^6;q^6)_\infty}{(q;q)_\infty}. \label{W-S1-result}
\end{align}

Let $(\widetilde{\alpha}_n^{(0)}(q;q),\widetilde{\beta}_n^{(0)}(q;q))$ be a Bailey pair with
\begin{align}
\widetilde{\alpha}_{n}^{(0)}(q;q)=q^{n^2+n}.
\end{align}
Applying \eqref{Bailey's lemma-1.1} to it, we obtain the Bailey pair
\begin{equation}\label{W-BP-b-1}
\widetilde{\alpha}_{n}^{(1)}(q;q)=q^{2n^2+2n}, \quad \widetilde{\beta}_{n}^{(1)}(q;q)=\sum_{k=0}^{n}\frac{q^{k^2+k}}{(q;q)_{n-k}}\widetilde{\beta}_{k}^{(0)}(q;q).
\end{equation}
Applying  \eqref{McLaughlin-a-a/q} to \eqref{W-BP-b-1}, we obtain the Bailey pair:
\begin{equation}\label{W-BP-b-2}
\begin{split}
\widetilde{\alpha}_0^{(2)}(1;q)&=\widetilde{\alpha}_0^{(1)}(1;q)=1, \\
\widetilde{\alpha}_n^{(2)}(1;q)&=\frac{1-q}{{1-q^{2n+1}}}\widetilde{\alpha}_n^{(1)}(q;q)-\frac{q^{2n-1}(1-q)}{1-q^{2n-1}}\widetilde{\alpha}_{n-1}^{(1)}(q;q), \quad n\geq 1, \\
\widetilde{\beta}_n^{(2)}(1;q)&=\widetilde{\beta}_n^{(1)}(q;q).
\end{split}
\end{equation}
Simplifying the expression of $\widetilde{\alpha}_n^{(2)}(1;q)$, we find that for $n\geq 1$,
\begin{align}
    \widetilde{\alpha}_n^{(2)}(1;q)=\frac{(1-q)q^{2n^2+2n}}{1-q^{2n+1}}-\frac{(1-q)q^{2n^2-1}}{1-q^{2n-1}}.
\end{align}
We can now evaluate the second sum in \eqref{CW-thm6.2-ij3-in-3-6}. Applying \eqref{Bailey's lemma-1.2} with the Bailey pair $(\widetilde{\alpha}_n^{(2)}(1,q),\widetilde{\beta}_n^{(2)}(1,q))$, we deduce that
\begin{align}
&T_1(q,q^{-1},q^{-1};q)=2\sum_{n_{2}=0}^{\infty}\sum_{n_{1}=0}^{n_{2}}\frac{q^{n_{1}^{2}+n_{1}+n_{2}^{2}}}{(q;q)_{n_{2}-n_{1}}} \times\sum_{r=0}^{n_{1}}\frac{q^{r^2+r}}{(q;q)_{n_{1}-r}(q;q)_{n_{1}+r+1}} \nonumber \\
&=\frac{2}{1-q}\sum_{n=0}^\infty q^{n^2}\widetilde{\beta}_n^{(2)}(1;q)=\frac{1}{(1-q)(q;q)_\infty} \sum_{n=0}^\infty q^{n^2}\widetilde{\alpha}_n^{(2)}(1;q) \nonumber \\
&=\frac{2}{(q;q)_\infty} \Big(\sum_{n=0}^\infty \frac{q^{3n^2+2n}}{1-q^{2n+1}}-\sum_{n=1}^\infty \frac{q^{3n^2-1}}{1-q^{2n-1}}\Big) \nonumber \\
&=\frac{2}{(q;q)_\infty} \Big(\sum_{n=0}^\infty \frac{q^{3n^2+2n}}{1-q^{2n+1}}-\sum_{n=0}^\infty \frac{q^{3(n+1)^2-1}}{1-q^{2n+1}}\Big) \nonumber \\
&=\frac{2}{(q;q)_\infty} \sum_{n=0}^\infty q^{3n^2+2n}(1+q^{2n+1}) =\frac{2}{(q;q)_\infty} \sum_{n=-\infty}^\infty q^{3n^2+2n}\nonumber \\
&=2\frac{(-q,-q^5,q^6;q^6)_\infty}{(q;q)_\infty}.\label{W-S2-result}
\end{align}
Substituting \eqref{W-S1-result} and \eqref{W-S2-result} with $q$ replaced by $q^2$ into \eqref{CW-thm6.2-ij3-in-3-6}, and then using the method in \cite{Frye-Garvan} to verify theta function identities, we prove \eqref{thm3-12}.
\end{proof}

\subsection{Lift-dual of the CW Example 2}
Recall the modular triples from \cite[(4.3)]{Cao-Wang2024}:
\begin{equation}\label{eq-Cao-Wang-unsolved}
\begin{split}
&A=\begin{pmatrix} 2 & -1 & -1 \\ -1 & 2 & 0 \\ -1 & 0 & 2 \end{pmatrix},~~ B_1=\begin{pmatrix} -1 \\ 1 \\ 0 \end{pmatrix},~~ B_2=\begin{pmatrix} 0 \\ 0 \\ 0 \end{pmatrix}, ~~ B_3=\begin{pmatrix} -1 \\ 1 \\ 1 \end{pmatrix}, \\
& B_4=\begin{pmatrix} -1 \\ 0 \\ 1 \end{pmatrix}, ~~ C_1=\frac{3}{14}, ~~ C_2=-\frac{1}{14}, ~~ C_3=\frac{5}{14},~~C_4=\frac{3}{14}.
\end{split}
\end{equation}
Here $(B_4,C_4)$ is obtained from $(B_1,C_1)$ for free due to the symmetry of $n_2$ and $n_3$ in the quadratic form $\frac{1}{2}n^\mathrm{T}An$ and was omitted in \cite[(4.3)]{Cao-Wang2024}. Applying the lift operators $\mathcal{L}_i$ ($i=1,2,3$) generate rank four modular triples and their dual are likely to be modular.

First, we record the lift and lift-dual via $\mathcal{L}_1$  in Table \ref{tab-A-exam5-L1}.

{\small
\begin{table}[htbp]
\centering
\caption{Lift and lift-dual via $\mathcal{L}_1$ of the modular triples $(A,B,C)$  in \eqref{eq-Cao-Wang-unsolved}}
    \label{tab-A-exam5-L1}
\begin{tabular}{|c|cccc|}
\hline 
\padedvphantom{I}{3ex}{3ex}
$\mathcal{L}_1(A)$ &  \multicolumn{4}{c|}{$\left(\begin{smallmatrix}  2 & -1 & -1 & -2 \\ -1 & 2 & 1 & 2 \\ -1 & 1 & 2 & 2 \\ -2 & 2 & 2 & 4 \end{smallmatrix}\right)$}   \\
\hline 
\padedvphantom{I}{3ex}{3ex}
$\mathcal{L}_1(B)$ & $\left(\begin{smallmatrix}
    -1 \\ 1 \\  0 \\ 1
\end{smallmatrix}\right)$ & $\left(\begin{smallmatrix}
    0 \\ 0 \\ 0 \\ 0
\end{smallmatrix}\right)$ & $\left(\begin{smallmatrix}
    -1 \\ 1 \\ 1 \\ 2
\end{smallmatrix}\right)$ & $\left(\begin{smallmatrix}
    -1 \\ 0 \\ 1 \\ 1
\end{smallmatrix}\right)$ \\
 \padedvphantom{I}{1ex}{1ex}
$\mathcal{L}_1(C)$ & {\tiny ${3}/{14}$} & {\tiny $-{1}/{14}$} & {\tiny ${5}/{14}$}  & {\tiny ${3}/{14}$} \\
\hline 
\padedvphantom{I}{3.5ex}{3.5ex}
 $\mathcal{DL}_1(A)$ & \multicolumn{4}{c|}{$\left(\begin{smallmatrix}
    1 &0  & 0 & 1/2 \\
    0 & 1 & 0 & -1/2 \\
    0 & 0 & 1 & -1/2 \\
    1/2 & -1/2 & -1/2 & 1
\end{smallmatrix}\right)$} \\
\hline 
\padedvphantom{I}{3ex}{3ex}
$\mathcal{DL}_1(B)$ & $\left(\begin{smallmatrix}
    -1/2 \\ 1/2 \\ -1/2 \\ 0
\end{smallmatrix}\right)$ & $\left(\begin{smallmatrix}
    0 \\ 0 \\ 0 \\ 0
\end{smallmatrix}\right)$ & $\left(\begin{smallmatrix}
    0 \\ 0 \\ 0 \\ 1/2
\end{smallmatrix}\right)$ & $\left(\begin{smallmatrix}
    -1/2 \\ -1/2 \\ 1/2 \\ 0
\end{smallmatrix}\right)$ \\
 \padedvphantom{I}{1ex}{1ex}
 $\mathcal{DL}_1(C)$&  {\tiny ${5}/{42}$} & {\tiny $-{2}/{21}$} & {\tiny $-{1}/{42}$} & {\tiny ${5}/{42}$} \\
 \hline
 \end{tabular}
\end{table}
}

We confirm the modularity of the triples $\mathcal{DL}_1(A,B,C)$ by establishing the following identities.
\begin{theorem}\label{thm-add1}
Let
\begin{align}
   F(u,v,w,t;q^2):=\sum_{i,j,k,l\geq 0} \frac{u^iv^jw^kt^lq^{i^2+j^2+k^2+l^2+(i-j-k)l}}{(q^2;q^2)_i(q^2;q^2)_j(q^2;q^2)_k(q^2;q^2)_l}.
\end{align}
We have
\begin{align}
F(u,v,w,t;q^2)&=F(u,w,v,t;q^2), \label{add-eq-F-symmetry} \\
    F(1,1,1,1;q^2)&=\frac{J_2^6J_{14}J_{28}^5}{J_1^3J_4^3J_{2,28}J_{3,28}J_{7,28}J_{8,28}J_{11,28}J_{12,28}}+4q\frac{J_4^3J_{8,28}}{J_2^4}, \label{add-eq-F-1} \\
    F(1,1,1,q;q^2)&=\frac{J_2^6J_{14}J_{28}^5}{J_1^3J_4^3J_{4,28}J_{5,28}J_{6,28}J_{7,28}J_{8,28}J_{9,28}}+4q^2\frac{J_4^3J_{4,28}}{J_2^4}, \label{add-eq-F-2} \\
    F(q^{-1},q,q^{-1},1;q^2)&=4\frac{J_4^3J_{12,28}}{J_2^4}+\frac{J_2^6J_{14}J_{28}^5}{J_1^3J_4^3J_{1,28}J_{4,28}J_{7,28}J_{10,28}J_{12,28}J_{13,28}}. \label{add-eq-F-3}
\end{align}
\end{theorem}
\begin{proof}
The identity \eqref{add-eq-F-symmetry} is trivial.
Summing over $i,j,k$ using \eqref{Euler2}, we deduce that
\begin{align}
    F(u,v,w,t;q^2)=\sum_{l=0}^\infty \frac{q^{l^2}t^l}{(q^2;q^2)_l}(-uq^{1+l};q^2)_\infty (-vq^{1-l};q^2)_\infty (-wq^{1-l};q^2)_\infty. \label{add-thm-start}
\end{align}

(1) When $u=v=w=t=1$, we have
\begin{align}
    F(1,1,1,1;q^2)=\sum_{l=0}^\infty \frac{q^{l^2}}{(q^2;q^2)_l}(-q^{1+l};q^2)_\infty(-q^{1-l};q^2)_\infty^2=S_0(q)+S_1(q), \label{add-1-S0S1}
\end{align}
where $S_0(q)$ and $S_1(q)$ correspond to the sum with $l$ even and odd, respectively. We have
\begin{align}
   & S_0(q)=\sum_{l=0}^\infty \frac{q^{4l^2}}{(q^2;q^2)_{2l}}(-q^{1+2l};q^2)_\infty (-q^{1-2l};q^2)_\infty^2 \nonumber \\
    &=(-q;q^2)_\infty^3 \sum_{l=0}^\infty \frac{q^{2l^2}(-q;q^2)_l}{(q^2;q^2)_{2l}}=(-q;q^2)_\infty^3 \sum_{l=0}^\infty \frac{q^{2l^2}}{(q;q^2)_l(q^4;q^4)_l} \nonumber  \\
    &=(-q;q^2)_\infty^3 \frac{(-q^3,q^4,-q^7;-q^7)_\infty}{(q^2;q^2)_\infty}. \label{add-1-S0}
\end{align}
Here for the last equality we used \eqref{S33} with $q$ replaced by $-q$.

Similarly, we have
\begin{align}
    &S_1(q)=\sum_{l=0}^\infty \frac{q^{4l^2+4l+1}}{(q^2;q^2)_{2l+1}}(-q^{2l+2};q^2)_\infty (-q^{-2l};q^2)_\infty^2 \nonumber \\
    &=4q(-q^2;q^2)_\infty^3 \sum_{l=0}^\infty \frac{q^{2l^2+2l}(-q^2;q^2)_l}{(q^2;q^2)_{2l+1}}=4q(-q^2;q^2)_\infty^3\sum_{l=0}^\infty \frac{q^{2l^2+2l}}{(q^2;q^2)_l(q^2;q^4)_{l+1}}  \nonumber \\
    &=4q(-q^2;q^2)_\infty^3\frac{(q^8,q^{20},q^{28};q^{28})_\infty}{(q^2;q^2)_\infty}. \label{add-1-S1}
\end{align}
Here for the last equality we used \eqref{S60} with $q$ replaced by $q^2$.

Substituting \eqref{add-1-S0} and \eqref{add-1-S1} into \eqref{add-1-S0S1}, we obtain \eqref{add-eq-F-1}.

(2) Setting $(u,v,w,t)=(1,1,1,q)$ in \eqref{add-thm-start}, we deduce that
\begin{align}
    F(1,1,1,q;q^2)=\sum_{l=0}^\infty \frac{q^{l^2+l}}{(q^2;q^2)_l}(-q^{1+l};q^2)_\infty (-q^{1-l};q^2)_\infty^2 =S_0(q)+S_1(q), \label{add-2-S0S1}
\end{align}
where $S_0(q)$ and $S_1(q)$ correspond to the sum with $l$ even and odd, respectively.

We have
\begin{align}
   & S_0(q)=\sum_{l=0}^\infty \frac{q^{4l^2+2l}}{(q^2;q^2)_{2l}}(-q^{1+2l};q^2)_\infty (-q^{1-2l};q^2)_\infty^2 \nonumber \\
    &=(-q;q^2)_\infty^3 \sum_{l=0}^\infty \frac{q^{2l^2+2l}(-q;q^2)_l}{(q^2;q^2)_{2l}} =(-q;q^2)_\infty^3 \sum_{l=0}^\infty \frac{q^{2l^2+2l}}{(q;q^2)_l(q^4;q^4)_l} \nonumber \\
    &=(-q;q^2)_\infty^3 \frac{(q^2,-q^5,-q^7;-q^7)_\infty}{(q^2;q^2)_\infty}. \label{add-2-S0}
\end{align}
Here for the last equaltiy we used \eqref{S32} with $q$ replaced by $-q$.

Similarly, we have
\begin{align}
    &S_1(q)=\sum_{l=0}^\infty \frac{q^{4l^2+6l+2}}{(q^2;q^2)_{2l+1}}(-q^{2+2l};q^2)_\infty (-q^{-2l};q^2)_\infty^2 \nonumber \\
    &=4q^2(-q^2;q^2)_\infty^3 \sum_{l=0}^\infty \frac{q^{2l^2+4l}(-q^2;q^2)_l}{(q^2;q^2)_{2l+1}} =4q^2(-q^2;q^2)_\infty^3 \sum_{l=0}^\infty \frac{q^{2l^2+4l}}{(q^2;q^2)_l(q^2;q^4)_{l+1}} \nonumber \\
    &=4q^2(-q^2;q^2)_\infty^3\frac{(q^4,q^{24},q^{28};q^{28})_\infty}{(q^2;q^2)_\infty}. \label{add-2-S1}
\end{align}
Here for the last equality we used \eqref{S59} with $q$ replaced by $q^2$.

Substituting \eqref{add-2-S0} and \eqref{add-2-S1} into \eqref{add-2-S0S1}, we obtain \eqref{add-eq-F-2}.

(3) Setting $(u,v,w,t)=(q^{-1},q,q^{-1},1)$ in \eqref{add-thm-start}, we deduce that
\begin{align}
    F(q^{-1},q,q^{-1},1;q^2)&=\sum_{l=0}^\infty \frac{q^{l^2}}{(q^2;q^2)_l}(-q^l;q^2)_\infty(-q^{2-l};q^2)_\infty(-q^{-l};q^2)_\infty \nonumber \\
    &=S_0(q)+S_1(q), \label{add-3-S0S1}
\end{align}
where $S_0(q)$ and $S_1(q)$ correspond to the sum with $l$ even and odd, respectively.

We have
\begin{align}
    &S_0(q)=\sum_{l=0}^\infty \frac{q^{4l^2}}{(q^2;q^2)_{2l}}(-q^{2l};q^2)_\infty (-q^{2-2l};q^2)_\infty (-q^{-2l};q^2)_\infty \nonumber \\
    &=4(-q^2;q^2)_\infty^3 \sum_{l=0}^\infty \frac{q^{2l^2}(-q^2;q^2)_l}{(q^2;q^2)_{2l}}=4(-q^2;q^2)_\infty^3 \sum_{l=0}^\infty \frac{q^{2l^2}}{(q^2;q^2)_l(q^2;q^4)_l}\nonumber \\
    &=4(-q^2;q^2)_\infty^3 \frac{(q^{12},q^{16},q^{28};q^{28})_\infty}{(q^2;q^2)_\infty}. \label{add-3-S0}
\end{align}
Here for the last equality we used \eqref{S61} with $q$ replaced by $q^2$.

Similarly, we have
\begin{align}
&S_1(q)=\sum_{l=0}^\infty \frac{q^{4l^2+4l+1}}{(q^2;q^2)_{2l+1}}(-q^{2l+1};q^2)_\infty (-q^{1-2l};q^2)_\infty (-q^{-1-2l};q^2)_\infty \nonumber \\
&=(-q;q^2)_\infty^3\sum_{l=0}^\infty \frac{q^{2l^2+2l}(-q;q^2)_{l+1}}{(q^2;q^2)_{2l+1}} =(-q;q^2)_\infty^3 \sum_{l=0}^\infty \frac{q^{2l^2+2l}}{(q;q^2)_{l+1}(q^4;q^4)_l} \nonumber \\
&=(-q;q^2)_\infty^3 \frac{(-q,q^6,-q^7;-q^7)_\infty}{(q^2;q^2)_\infty}. \label{add-3-S1}
\end{align}
Here for the last equality we used \eqref{S31} with $q$ replaced by $-q$.

Substituting \eqref{add-3-S0} and \eqref{add-3-S1} into \eqref{add-3-S0S1}, we obtain \eqref{add-eq-F-3}.
\end{proof}

Note that the lifts via $\mathcal{L}_2$ and $\mathcal{L}_3$ generate essentially the same set of rank four Nahm sums (simply interchange the variables $n_2$ and $n_3$). Hence it suffices to discuss $\mathcal{L}_2$ and we record the lift and lift-dual via $\mathcal{L}_2$ in Table \ref{tab-A-exam5-L2}. At this moment, we do not know how to prove the modularity of the lift-dual. It seems that at least two infinite products will be involved to express the Nahm sums associated with the triples $\mathcal{DL}_2(A,B,C)$.

{\small
\begin{table}[htbp]
\centering
\caption{Lift and lift-dual via $\mathcal{L}_2$ of the modular triples $(A,B,C)$  in \eqref{eq-Cao-Wang-unsolved}}
    \label{tab-A-exam5-L2}
\begin{tabular}{|c|cccc|}
\hline 
\padedvphantom{I}{3ex}{3ex}
$\mathcal{L}_2(A)$     &  \multicolumn{4}{c|}{$\left(\begin{smallmatrix}  2 & -1 & 0 & 1 \\ -1 & 2 & 0 & -1 \\ 0 & 0 & 2 & 1 \\ 1 & -1 & 1 & 2 \end{smallmatrix}\right)$}   \\
    \hline 
    \padedvphantom{I}{3ex}{3ex}
 $\mathcal{L}_2(B)$  & $\left(\begin{smallmatrix} -1 \\ 1 \\ 0 \\ -1 \end{smallmatrix}\right)$ &  $\left(\begin{smallmatrix} 0 \\ 0 \\ 0 \\ 0 \end{smallmatrix}\right)$ & $\left(\begin{smallmatrix} -1 \\ 1 \\ 1 \\ 0 \end{smallmatrix}\right)$ & $\left(\begin{smallmatrix}
     -1 \\ 0 \\ 1 \\ 0
 \end{smallmatrix}\right)$  \\
  \padedvphantom{I}{1ex}{1ex}
 $\mathcal{L}_2(C)$   & {\tiny ${3}/{14}$} & {\tiny $-{1}/{14}$} & {\tiny ${5}/{14}$}  & {\tiny ${3}/{14}$}  \\
    \hline 
    \padedvphantom{I}{3.5ex}{3.5ex}
     $\mathcal{DL}_2(A)$ &
      \multicolumn{4}{c|}{$\left(\begin{smallmatrix} 4/5 & 1/5 & 1/5 & -2/5 \\ 1/5 & 4/5 & -1/5 & 2/5 \\ 1/5 & -1/5 & 4/5 & -3/5 \\ -2/5 & 2/5 & -3/5 & 6/5 \end{smallmatrix}\right)$} \\
    \hline 
    \padedvphantom{I}{3.5ex}{3.5ex}
    $\mathcal{DL}_2(B)$
    & $\left(\begin{smallmatrix} -1/5 \\ 1/5 \\ 1/5 \\ -2/5  \end{smallmatrix}\right)$ &  $\left(\begin{smallmatrix} 0 \\ 0 \\ 0 \\ 0 \end{smallmatrix}\right)$ & $\left(\begin{smallmatrix} -2/5 \\ 2/5 \\ 2/5 \\ 1/5 \end{smallmatrix}\right)$  & $\left(\begin{smallmatrix}
        -3/5 \\ -2/5 \\ 3/5 \\ -1/5
    \end{smallmatrix}\right)$  \\
     \padedvphantom{I}{1ex}{1ex}
  $\mathcal{DL}_2(C)$
    & {\tiny ${2}/{105}$} & {\tiny $-{2}/{21}$} & {\tiny ${8}/{105}$}  & {\tiny ${23}/{105}$} \\
    \hline
\end{tabular}
\end{table}
}

\subsection{Lift-dual of the CW Example 3}
We list the modular triples in \cite[Eq.\ (1.15)]{Cao-Wang2024}) and their lift and lift-dual in Table \ref{tab-1}. Note that applying the lifting operators $\mathcal{L}_i$ ($i=1,2,3$) to the modular triples in Table \ref{tab-1} will generate the same set of modular triples,  and hence we only list the lift by $\mathcal{L}_2$.

{\small
\begin{table}[htbp]
\centering
\caption{Modular triples in \cite[Eq.\ (1.15)]{Cao-Wang2024} and their lift and lift-dual}\label{tab-1}
\begin{tabular}{|c|cccc|}
  \hline 
  \padedvphantom{I}{3ex}{3ex}
  $A$ & \multicolumn{4}{c|}{$\left(\begin{smallmatrix} 3/2 & -1/2 & -1/2  \\  -1/2 & 3/2 & -1/2  \\ -1/2 & -1/2 & 3/2  \\  \end{smallmatrix}\right)$} \\
  \hline 
  \padedvphantom{I}{3ex}{3ex}
        $B$ & $\left(\begin{smallmatrix} 1/2 \\ -1/2 \\ -1/2 \end{smallmatrix}\right)$ & $ \left(\begin{smallmatrix} -1/2 \\ -1/2 \\ 1/2 \end{smallmatrix}\right)$ & $\left(\begin{smallmatrix} -1/2 \\ 1/2 \\ -1/2 \end{smallmatrix}\right)$ & $\left(\begin{smallmatrix} 0 \\ 0 \\ 0 \end{smallmatrix}\right)$ \\
         \padedvphantom{I}{1ex}{1ex}
        $C$ & {\tiny $-3/40$} & {\tiny $-3/40$} & {\tiny $-3/40$ }& {\tiny $3/40$} \\
    \hline 
    \padedvphantom{I}{4ex}{4ex}
  $\mathcal{L}_2(A)$ & \multicolumn{4}{c|}{$\left(\begin{smallmatrix} 3/2 & 1/2 & -1/2 & 1 \\  1/2 & 3/2 & -1/2 & 1 \\ -1/2 & -1/2 & 3/2 & -1 \\ 1 & 1 & -1 & 2  \end{smallmatrix}\right)$} \\
   \hline 
   \padedvphantom{I}{4ex}{4ex}
        $\mathcal{L}_2(B)$ & $\left(\begin{smallmatrix} 1/2 \\ -1/2 \\ -1/2 \\ 0 \end{smallmatrix}\right)$ & $ \left(\begin{smallmatrix} -1/2 \\ 1/2 \\ -1/2 \\ 0 \end{smallmatrix}\right)$ & $\left(\begin{smallmatrix} -1/2 \\ -1/2 \\ 1/2 \\ -1 \end{smallmatrix}\right)$ & $\left(\begin{smallmatrix} 0 \\ 0 \\ 0 \\ 0 \end{smallmatrix}\right)$ \\
         \padedvphantom{I}{1ex}{1ex}
        $\mathcal{L}_2(C)$ & {\tiny $-3/40$} & {\tiny $-3/40$} & {\tiny $-3/40$} & {\tiny $3/40$} \\
    \hline 
    \padedvphantom{I}{4ex}{4ex}
    $\mathcal{DL}_2(A)$ & \multicolumn{4}{c|}{$\left(\begin{smallmatrix} 1 & 0 & 0 & -1/2 \\  0 & 1 & 0 & -1/2 \\ 0 & 0 & 1 & 1/2 \\ -1/2 & -1/2 & 1/2 & 5/4  \end{smallmatrix}\right)$} \\
   \hline 
   \padedvphantom{I}{4ex}{4ex}
        $\mathcal{DL}_2(B)$ & $\left(\begin{smallmatrix} -1/2 \\ 1/2 \\ -1/2 \\ -1/4 \end{smallmatrix}\right)$ & $ \left(\begin{smallmatrix} 1/2 \\ -1/2 \\ -1/2 \\ -1/4  \end{smallmatrix}\right)$ & $\left(\begin{smallmatrix} 0 \\ 0 \\ 0 \\ -1/2\end{smallmatrix}\right)$ & $\left(\begin{smallmatrix} 0 \\ 0 \\ 0 \\ 0 \end{smallmatrix}\right)$ \\
         \padedvphantom{I}{1ex}{1ex}
        $\mathcal{DL}_2(C)$ & {\tiny $2/15$} & {\tiny $2/15$} & {\tiny $1/120$} & {\tiny $-11/120$} \\
    \hline
\end{tabular}
\end{table}
}

Since $n_1$ and $n_2$ are symmetric to each other in the quadratic form generated by $\mathcal{DL}_2(A)$, there are essentially three different Nahm sums associated with the triples $\mathcal{DL}_2(A,B,C)$ in Table \ref{tab-1}. We now justify their modularity by establishing the following identities.

\begin{theorem}\label{CW-thm5.1-in}
Let
\begin{align}
F(u,v,w,t;q^2):=\sum_{i,j,k,l\geq 0} \frac{u^iv^jw^kt^lq^{i^2+j^2+k^2+5l^2/4-il-jl+kl}}{(q^2;q^2)_i(q^2;q^2)_j(q^2;q^2)_k(q^2;q^2)_l}.
\end{align}
We have
\begin{align}
&F(u,v,w,t;q^2)=F(v,u,w,t;q^2), \label{thm2-0} \\
 &F(q^{-1},q,q^{-1},q^{-1/2};q^2)=4\frac{J_{4}^{3}J_{8,20}}{J_{2}^{4}}
+q^{-\frac{1}{4}}\frac{J_{2}^{7}J_{3,10}J_{4,20}}{J_{1}^{4}J_{4}^{4}J_{20}}, \label{thm2-1} \\
&F(1,1,1,q^{-1};q^2)=\frac{J_{2}^{7}J_{3,10}J_{4,20}}{J_{1}^{4}J_{4}^{4}J_{20}}
+4q^{\frac{1}{4}}\frac{J_{4}^{3}J_{8,20}}{J_{2}^{4}}, \label{thm2-2}\\
&F(1,1,1,1;q^2) =\frac{J_{2}^{7}J_{1,10}J_{8,20}}{J_{1}^{4}J_{4}^{4}J_{20}}
+4q^{\frac{5}{4}}\frac{J_{4}^{3}J_{4,20}}{J_{2}^{4}}. \label{thm2-3}
\end{align}
\end{theorem}
\begin{proof}
The identity \eqref{thm2-0} is trivial. Summing over $i,j,k$ using \eqref{Euler1} and then splitting the sum into two parts according to the parity of $l$, we have
\begin{align}
&F(u,v,w,t;q^2)=
\sum_{l\geq 0} \frac{t^lq^{\frac{5}{4}l^2}}{(q^2;q^2)_l}(-uq^{1-l},-vq^{1-l},-wq^{1+l};q^2)_{\infty} \nonumber\\
&=(-uq,-vq,-wq;q^2)_{\infty} \sum_{l\geq 0} \frac{t^{2l}q^{5l^2}(-uq^{1-2l},-vq^{1-2l};q^2)_{l}}{(q^2;q^2)_{2l}(-wq;q^2)_{l}} \nonumber\\
&\quad +(-u,-v,-wq^2;q^2)_{\infty} \sum_{l\geq 0} \frac{t^{2l+1}q^{5l^2+5l+\frac{5}{4}}(-uq^{-2l},-vq^{-2l};q^2)_{l}}{(q^2;q^2)_{2l+1}(-wq^2;q^2)_{l}} \nonumber\\
&=(-uq,-vq,-wq;q^2)_{\infty} \sum_{l\geq 0} \frac{u^lv^lt^{2l}q^{3l^2}(-q/u,-q/v;q^2)_{l}}{(q^2;q^2)_{2l}(-wq;q^2)_{l}} \nonumber\\
&\quad +(-u,-v,-wq^2;q^2)_{\infty} \sum_{l\geq 0} \frac{u^lv^lt^{2l+1}q^{3l^2+3l+\frac{5}{4}}(-q^{2}/u,-q^{2}/v;q^2)_{l}}{(q^2;q^2)_{2l+1}(-wq^2;q^2)_{l}}.
\end{align}
Letting $v=1/u$, we have
\begin{align}\label{thm-2-proof-start}
&F(u,1/u,w,t;q^2)=(-uq,-q/u,-wq;q^2)_{\infty} \sum_{l\geq 0} \frac{t^{2l}q^{3l^2}(-q/u,-qu;q^2)_{l}}{(q^2;q^2)_{2l}(-wq;q^2)_{l}} \nonumber\\
& \qquad +(-u,-1/u,-wq^2;q^2)_{\infty} \sum_{l\geq 0} \frac{t^{2l+1}q^{3l^2+3l+\frac{5}{4}}(-q^{2}/u,-q^{2}u;q^2)_{l}}{(q^2;q^2)_{2l+1}(-wq^2;q^2)_{l}}.
\end{align}

(1) Setting $(u,w,t)=(q^{-1},q^{-1},q^{-1/2})$ in \eqref{thm-2-proof-start}, we deduce that
\begin{align}
&F(q^{-1},q,q^{-1},q^{-1/2};q^2)=(-1,-q^{2},-1;q^2)_{\infty} \sum_{l\geq 0} \frac{q^{3l^2-l}(-q^2;q^2)_{l}}{(q^2;q^2)_{2l}} \nonumber \\
&\qquad +(-q^{-1},-q,-q;q^2)_{\infty} \sum_{l\geq 0} \frac{q^{3l^2+2l+\frac{3}{4}}(-q^{3};q^2)_{l}}{(q^2;q^2)_{2l+1}}.
\end{align}
Substituting \eqref{S. 46} and \eqref{A. 97} into it, we obtain \eqref{thm2-1}.

(2) Setting $(u,w,t)=(1,1,q^{-1})$ in \eqref{thm-2-proof-start}, we deduce that
\begin{align}
&F(1,1,1,q^{-1};q^2)=(-q,-q,-q;q^2)_{\infty} \sum_{l\geq 0} \frac{q^{3l^2-2l}(-q;q^2)_{l}}{(q^2;q^2)_{2l}} \nonumber \\
&\qquad +(-1,-1,-q^2;q^2)_{\infty} \sum_{l\geq 0} \frac{q^{3l^2+l+\frac{1}{4}}(-q^{2};q^2)_{l}}{(q^2;q^2)_{2l+1}}.
\end{align}
Substituting \eqref{S. 95} and \eqref{S. 62} into it, we obtain \eqref{thm2-2}.

(3) Setting $(u,w,t)=(1,1,1)$ in \eqref{thm-2-proof-start}, we deduce that
\begin{align}
&F(1,1,1,1;q^2) =(-q,-q,-q;q^2)_{\infty} \sum_{l\geq 0} \frac{q^{3l^2}(-q;q^2)_{l}}{(q^2;q^2)_{2l}} \nonumber \\
&\qquad +(-1,-1,-q^2;q^2)_{\infty} \sum_{l\geq 0} \frac{q^{3l^2+3l+\frac{5}{4}}(-q^{2};q^2)_{l}}{(q^2;q^2)_{2l+1}}.
\end{align}
Substituting \eqref{A. 100} and \eqref{S. 63} into it, we obtain \eqref{thm2-3}.
\end{proof}

To summarize, we have confirmed the modularity of four sets of rank four Nahm sums associated with the triples $\mathcal{DL}_i(A,B,C)$ listed in Tables \ref{Wang-3-Exam7}, \ref{tab-4},  \ref{tab-A-exam5-L1} and \ref{tab-1}.  For those we cannot prove at the moment, we propose the following conjecture.
\begin{conj}\label{conj-table}
The Nahm sums associated with the lift-dual triples $\mathcal{DL}_i(A,B,C)$  listed in Tables \ref{tab-A-exam1}--\ref{tab-A-exam4} and \ref{tab-A-exam5-L2} are all modular.
\end{conj}

\subsection*{Acknowledgements}
This work was supported by the National Key R\&D Program of China (Grant No.\ 2024YFA1014500).


\begin{thebibliography}{0}
\bibitem{Andrews1984} G.E. Andrews, Multiple $q$-series identities, Houston J. Math. 7(1) (1981), 11--22.

\bibitem{Andrews1998} G.E. Andrews, The Theory of Partitions, Addison-Wesley, 1976; Reissued Cambridge, 1998.

\bibitem{Berndt2006} B.C. Berndt, Number Theory in the Spirit of Ramanujan, AMS 2006.

\bibitem{Bressoud1980} D.M. Bressoud, Analytic and combinatorial generalizations of the Rogers--Ramanujan identities, Mem. Amer. Math. Soc. 24 (1980).

\bibitem{Bressoud2000} D.M. Bressoud, M.E.H. Ismail and D. Stanton, Change of base in Bailey pairs, Ramanujan J. 4 (2000), 435--453.

\bibitem{CGZ} F. Calegari, S. Garoufalidis and D. Zagier, Bloch groups, algebraic K-theory, units, and Nahm's conjecture, Ann. Sci. \'Ec. Norm. Sup\'er. (4) 56 (2023), no.\ 2, 383--426.


\bibitem{CMP} C. Calinescu, A. Milas and M. Penn, Vertex algebraic structure of principal subspaces of basic $A_{2n}^{(2)}$-modules, J. Pure Appl. Algebra 220 (2016), 1752--1784.

\bibitem{CRW} Z. Cao, H. Rosengren and L. Wang, On some double Nahm sums of Zagier, J. Combin. Theory Ser. A 202 (2024), Paper No. 105819.

\bibitem{Cao-Wang2024} Z. Cao and L. Wang, Some new modular rank three Nahm sums from a lift-dual operation, arXiv:2412.15767.

\bibitem{Feigin} I. Cherednik and B. Feigin, Rogers--Ramanujan type identities and Nil-DAHA, Adv. Math. 248 (2013), 1050--1088.


\bibitem{Frye-Garvan} J. Frye and F.G. Garvan, Automatic proof of theta-function identities, Elliptic integrals, elliptic functions and modular forms in quantum field theory, Texts Monogr. Symbol. Comput., Springer,
Cham, 2019, pp. 195--258.

\bibitem{Gasper-Rahman} G. Gasper and M. Rahman, Basic Hypergeometric Series, 2nd Edition, Encyclopedia of Mathematics and Its Applications, Vol. 96, Cambridge University Press, 2004.


\bibitem{Lovejoy2004} J. Lovejoy, A Bailey lattice. Proc. Am. Math. Soc. 132 (2004), 1507--1516.

\bibitem{McLaughlin2018} J. Mc Laughlin, Topics and methods in $q$-series, Monographs in Number Theory, 8, World Scientific Publishing Co. Pte. Ltd., Hackensack, NJ, 2018.

\bibitem{MW24} A. Milas and L. Wang, Modularity of Nahm sums for the tadpole diagram, Int. J. Number Theory 20 (1) (2024), 73--101.

\bibitem{Nahm1994} W. Nahm, Conformal field theory and the dilogarithm, In 11th International Conference on Mathematical Physics (ICMP-11) (Satelite colloquia: New Problems in General Theory of Fields and Particles), Paris, 1994, 662--667.

\bibitem{Nahmconf} W. Nahm, Conformal field theory, dilogarithms and three dimensional manifold, in ``Interface between Physics and Mathematics (Proceedings, Conference in Hangzhou, People's Republic of China, September 1993)'', eds. W. Nahm and J.-M. Shen, World Scientific, Singapore, 1994, 154--165.

\bibitem{Nahm07} W. Nahm, Conformal field theory and torsion elements of the Bloch group. In Frontiers in Number Theory, Physics, and Geometry II: On Conformal Field Theories, Discrete Groups and Renormalization, pp.\ 67--132. Springer, 2007.


\bibitem{Rogers}L.J. Rogers, Second memoir on the expansion of certain infinite products, Proc. London Math. Soc. 25 (1894), 318--343.

\bibitem{Rogers1937} L.J. Rogers, On two theorems of combinatory analysis and some allied identities,
Proc. London Math. Soc. 16 (1917), 315--336.

\bibitem{Shi-Wang} C. Shi and L. Wang, Modularity of tadpole Nahm sums in ranks 4 and 5, arXiv:2504.17737.

\bibitem{Sills}A.V. Sills, Finite Rogers--Ramanujan type identities, Electron. J. Combin. 10(1) (2003), \#R13.

\bibitem{Slater}L.J. Slater, Further identities of the Rogers--Ramanujan type, Proc. Lond. Math. Soc. (2) 54
(1) (1952), 147--167.

\bibitem{VZ}M. Vlasenko and S. Zwegers, Nahm's conjecture: asymptotic computations and counterexamples,
Commun. Number Theory Phys. 5(3) (2011), 617--642.


\bibitem{Wang-rank2} L. Wang, Identities on Zagier's rank two examples for Nahm's problem, Res. Math. Sci. 11, 49 (2024).

\bibitem{Wang-rank3}L. Wang, Explicit forms and proofs of Zagier's rank three examples for Nahm's problem, Adv. Math. 450 (2024), 109743.


\bibitem{Wang-counterexample} L. Wang, Counterexamples to some duality principles on Nahm sums, arXiv:2411.09701v3.


\bibitem{Zagier} D. Zagier, The dilogarithm function, in Frontiers in Number Theory, Physics and Geometry, II, Springer, 2007, 3--65.

\end{thebibliography}
\end{document}